\documentclass[11pt,twoside]{article}
\usepackage{amsmath,bm}
\usepackage{amssymb}
\usepackage{fancyhdr}
\usepackage{latexsym}
\usepackage{bbding}
\usepackage{mathrsfs}
\usepackage{exscale}
\usepackage{relsize}
\usepackage{wasysym}
\usepackage{cite}
\usepackage{multicol,graphics}
\usepackage{color}
\usepackage{lineno}

\usepackage{tabularx}
\usepackage{enumerate}
\usepackage[dvipdfm]{graphicx}
\usepackage{texdraw}

\usepackage{microtype}
\usepackage{caption}
\usepackage{subcaption}
\usepackage[labelformat=parens,labelsep=space,font=small]{subcaption}
\allowdisplaybreaks[4]
\def\vp{\varepsilon}
\def\p{\partial}

\usepackage{mathrsfs}

\usepackage{amsfonts,amssymb,amsmath}
\usepackage{epsfig}

\makeatletter
\@addtoreset{equation}{section}

\makeatother

\newtheorem{theo}{\bf Theorem}[section]
\newtheorem{lem}[theo]{\bf Lemma}

\newtheorem{defi}[theo]{\bf Definition}
\newtheorem{rem}[theo]{\bf Remark}

\newenvironment{proof}[1][Proof]{\noindent\textbf{#1.} }{\hfill $\Box$}

\newcommand{\R}{\mathbb{R}}

\def\bar{\overline}

\allowdisplaybreaks[3]

\let\oldequation\equation
\let\oldendequation\endequation

\renewenvironment{equation}
{\linenomathNonumbers\oldequation}
{\oldendequation\endlinenomath}

\allowdisplaybreaks \numberwithin{equation}{section}
 \makeatletter
\begin{document}

\title{V-shaped transition fronts of monotone bistable reaction-diffusion systems in exterior 
domains}

\author{Yang-Yang Yan$^{a, b}$\ \ and\ \ 
Wei-Jie Sheng$^{a,}$\thanks{Corresponding author (E-mail address: 
shengwj09@hit.edu.cn).} \\~\\
\footnotesize{$^a$ School of Mathematics, Harbin Institute of Technology}, \\
\footnotesize{Harbin, Heilongjiang, 150001, People's Republic of China}\\
\footnotesize{$^b$ Aix Marseille Univ, CNRS, I2M, Marseille, France}
}

\date{}
\maketitle

\begin{abstract}

This paper investigates the propagation phenomena of a monotone bistable reaction–diffusion 
system in an exterior domain of $\mathbb{R}^2$. 
By constructing suitable sub- and supersolutions, we establish the existence and monotonicity of an entire solution originating from a V-shaped traveling front. 
It is further shown that, under the complete propagation condition, this entire solution eventually recovers its  V-shaped profile as time tends to $\infty$ after passing  the obstacle. 
In particular, we  show that the entire solution is a V-shaped transition front whose global mean 
speed coincides with the  planar wave speed.



\end{abstract}


\section{Introduction} \label{introduction}
In this paper, we study the propagation phenomena of the following reaction-diffusion system:
\begin{align}\label{1.1}
\begin{cases}
u_t = D\Delta u + F(u), & t\in\mathbb{R},\ x\in\Omega = \mathbb{R}^2\setminus K,\\
\nu\cdot\nabla u = \mathbf{0}, & t\in\mathbb{R},\ x\in\partial\Omega = \partial K,
\end{cases}
\end{align}
where $u = (u_1,\cdots,u_m)$, $F(u) = (F_1(u),\cdots,F_m(u))$, $\mathbf{0} = (0,\cdots,0)\in\mathbb{R}^m$, $m\ge2$, $\nu = \nu(x)$ is the outward unit normal vector at the point $x$ on the boundary $\partial\Omega$, and the obstacle $K$ is a compact subset of $\mathbb{R}^2$ with smooth boundary.

For convenience, we  first introduce some notations.
 Let $\mathbf1= (1,\cdots,1)\in\R^m$. For any vectors $A=(a_1,\cdots,a_m)$ and $B=(b_1,\cdots,b_m)$, the symbol $A\ll B$ means $a_i<b_i$ for each $i=1,\cdots,m$, and $A\leq B$ means $a_i\leq b_i$ for each $i=1,\cdots,m$. The interval $[A,B]$ denotes the set of $V\in\mathbb R^m$ such that $A\leq V\leq B$, and $(A,B)$ denotes the set of $V\in\mathbb R^m$ such that $A\ll V\ll B$.

Throughout the paper, we assume that the following conditions are satisfied.
\begin{itemize}
  \item[(A1)] $D$ is a diagonal matrix of order $m$ with positive diagonal entries $D_i>0$ for $i=1,2,\ldots,m$.

  \item[(A2)] $F$ has two stable equilibria $\mathbf{0}$ and $\mathbf{1}$, that is, $F(\mathbf{0}) = F(\mathbf{1}) = \mathbf{0}$, and all eigenvalues of the Jacobian matrices $F'(\mathbf{0})$ and $F'(\mathbf{1})$ lie in the open left-half complex plane.

  \item[(A3)] There exist two positive vectors $R_0=(r_{01},\ldots,r_{0m})$, $R_1=(r_{11},\ldots,r_{1m})$ and two positive numbers $\lambda_0$, $\lambda_1$ such that 
$ F'(\mathbf{0}) R_0 \le -\lambda_0 R_0$ and $F'(\mathbf{1}) R_1 \le -\lambda_1 R_1$.

  \item[(A4)] The nonlinearity $F(u)$ is defined on an open domain $\mathcal{A} \supset [\mathbf{0},\mathbf{1}]$. Moreover, $F \in C^1([\mathbf{0},\mathbf{1}],\mathbb{R}^m)$ and satisfies
$ \frac{\partial F_i}{\partial u_j}(u) \ge 0 $ for all $u \in [\mathbf{0},\mathbf{1}]$ and 
$ 1 \le i \ne j \le m$.
\end{itemize}

The condition (A1) means that
the system \eqref{1.1} is not degenerate.
The condition {(A2)} indicates that \eqref{1.1} is a bistable system.
The condition (A3) implies the existence of two Perron–Frobenius directions $R_0$ and $R_1$ at the stable equilibria $\mathbf{0}$ and $\mathbf{1}$, which holds by the Perron–Frobenius theorem if $F'(\mathbf{0})$ and $F'(\mathbf{1})$ are irreducible.
The condition {(A4)} ensures that the system \eqref{1.1} is  cooperative
and allows us to apply the comparison principle within the order interval $[\mathbf0,\mathbf1]$.
 These conditions have been used to study the V-shaped traveling front of \eqref{1.1} in \cite{wangzhicheng2012}. The difference is that we no longer assume condition (H4) in \cite{wangzhicheng2012} because the sub- and supersolutions established in this paper are bounded from below by $\mathbf{0}$ and from above by $\mathbf{1}$.

In general, conditions (A1)-(A4) are not sufficient to guarantee the existence of a planar traveling front connecting $\mathbf{0}$ and $\mathbf{1}$ for system \eqref{1.1}. Therefore, we assume in this paper that when $\Omega = \mathbb{R}^2$ (with no boundary condition), system \eqref{1.1} admits a unique (up to shifts) planar traveling front of the form
\[
u(t,x) = \Phi(x\cdot e - ct) = \left(\Phi_1(x\cdot e - ct), \ldots, \Phi_m(x\cdot e - ct)\right),
\]
where $e \in \mathbb{S}^1$ is the propagation direction ($\mathbb{S}^1$ denotes the unit circle in $\mathbb{R}^2$), $c \in \mathbb{R}$ is the wave speed, and $\Phi: \mathbb{R} \to [\mathbf{0},\mathbf{1}]$ is the wave profile satisfying
\begin{equation}\label{p}
\begin{cases}
D_i \Phi_i''(\xi) + c \Phi_i'(\xi) + F_i(\Phi(\xi)) = 0, & \xi \in \mathbb{R},\\
\Phi_i'(\xi) < 0, & \xi \in \mathbb{R},\\
\Phi_i(-\infty) = 1, \ \Phi_i(\infty) = 0,
\end{cases}
\end{equation}
for each $i=1,\ldots,m$.
According to \cite[Chapter 3]{VVV}, there exist two positive constants $a$ and $b$ such that for each $i=1,\ldots,m$,
\begin{align}\label{1.a}
\begin{cases}
\Phi_i(\xi),\ |\Phi_i'(\xi)|,\ |\Phi_i''(\xi)| \leq a e^{-b\xi}, & \xi \ge 0,\\
1-\Phi_i(\xi),\ |\Phi_i'(\xi)|,\ |\Phi_i''(\xi)| \leq a e^{b\xi}, & \xi < 0.
\end{cases}
\end{align}
We further assume that there exists $\mathcal C>0$ such that, for each $i=1,\ldots,m$,
\begin{align}\label{impo1}
\Phi_i''(\xi) \le 0 \text{ for } \xi \in (-\infty,-\mathcal{C}], \quad 
\Phi_i''(\xi) \ge 0 \text{ for } \xi \in [\mathcal{C},+\infty).
\end{align}
In particular, if $F$ is of class $C^{1+\theta}$ on $\mathcal{A}$ for some $\theta \in (0,1)$, and the off-diagonal entries of the Jacobian matrices $F'(\mathbf{0})$ and $F'(\mathbf{1})$ are positive, then \eqref{impo1} is satisfied according to \cite[Theorem 5.5]{SW18}.
Planar traveling waves have been extensively studied due to their importance in modeling and analyzing processes in biology, epidemiology, chemistry and physics,
see \cite{Alex,conley,fife,Gardner,Kan-on,Kan-on2,OT,Tsai,VVV1,VVV} and references therein.

Besides planar traveling fronts, it has been established that \eqref{1.1} admits various types of nonplanar traveling fronts, including V-shaped fronts, pyramidal-shaped fronts and cylindrically symmetric fronts, see \cite{B1,B2,H2,H3,NT,NT1,T1,T2,T3,wangzhicheng2012,WB} and references therein. 
These nonplanar fronts exhibit more complex geometric structures and play an important role in understanding multidimensional propagation phenomena. 
In fact, the existence and stability of  V-shaped traveling fronts for the bistable system \eqref{1.1} in $\mathbb R^2$ were first proved by Wang \cite{wangzhicheng2012}. Call  $(x_1,x_2)$  the coordinate of $x\in\mathbb R^2$.
He showed that for each $s>c$, there exists a V-shaped traveling front
\begin{align}\label{VSHAPE}
u(t,x) = V(x_1,x_2-st) = V(y,z) = (V_1(y,z),\cdots,V_m(y,z))
\end{align}
of \eqref{1.1} with asymptotic lines $z=m_*|y|$ satisfying
\begin{equation}\label{Veq}
- D V_{yy} - D V_{zz} - s V_y - F(V) = \mathbf{0}, \quad (y,z) \in \mathbb R^2
\end{equation}
and
\begin{equation}\label{1.b}
\lim_{R\to\infty} \sup_{y^2+z^2 \ge R} \left| V(y,z) - \Phi\left(\frac{c}{s}(z-m_*|y|)\right) \right| = 0,
\end{equation}
where
\begin{align}\label{mstar}
m_* = \frac{\sqrt{s^2 - c^2}}{c}.
\end{align}
Furthermore, $V_z(y,z) \ll \mathbf{0}$ for all $(y,z) \in \mathbb R^2$. 
He also proved that the V-shaped traveling front is globally asymptotically stable. More precisely, if $v_0(y,z) \in C^2(\mathbb R^2, [\mathbf{0}, \mathbf{1}])$ satisfies
\begin{align*}
\lim_{R\to\infty} \sup_{y^2+z^2 > R^2} \left| v_0(y,z) - \Phi\left(\frac{c}{s}(z-m_*|y|)\right) \right| = 0,
\end{align*}
then the solution $v(t,y,z;v_0)$ of the Cauchy problem
\[
\begin{cases}
v_t - D v_{yy} - D v_{zz} - s v_y - F(v) = \mathbf{0}, & t>0, \ (y,z)\in \mathbb R^2,\\
v(0,y,z) = v_0(y,z), & (y,z)\in \mathbb R^2
\end{cases}
\]
satisfies
\begin{align}\label{stable}
\lim_{t\to\infty} \left\| v(t,\cdot,\cdot;v_0) - V(\cdot,\cdot) \right\|_{C(\mathbb R^2)} = 0.
\end{align}

Despite the variety of traveling fronts, they share a common feature in that they always converge uniformly to $\mathbf{0}$ or $\mathbf{1}$ far away from their level sets.
This simple but fundamental property motivates the introduction of a fully general concept, namely, transition fronts. For a detailed description, we refer to \cite{BH1,BH2,shen}. 
We now recall the definition of transition fronts. For any two subsets $A$ and $B$ in $\overline\Omega$, we set
\begin{equation*}
d_\Omega(A,B)=\inf\{d_\Omega(x,y) : (x, y)\in A\times B\}, \quad 
d_\Omega(x,A)=d_\Omega(\{x\},A),
\end{equation*}
where $d_\Omega$ is the geodesic distance in $\overline\Omega$. Consider two families of nonempty unbounded open subsets $\{\Omega^\pm_t\}_{t\in\R}\subset\R^2$ satisfying
\begin{equation}\label{subset1}
\forall t\in\mathbb R,\ \ 
\begin{cases}
\Omega_t^-\cap\Omega_t^+ = \emptyset,\
\partial\Omega_t^+\cap\Omega=\partial\Omega_t^-\cap\Omega =: \Gamma_t,\
\Omega_t^- \cup \Gamma_t \cup \Omega_t^+ = \Omega,\\
\sup\{d_\Omega(x,\Gamma_t) : x\in\Omega_t^-\} = \sup\{d_\Omega(x,\Gamma_t) : x\in\Omega_t^+\} = \infty
\end{cases}
\end{equation}
and
\begin{equation}\label{subset2}
\begin{cases}
\inf\left\{\sup\{d_\Omega(y,\Gamma_t) : y\in\Omega_t^+,\, d_\Omega(x,y)\leq r\} : t\in\R, x\in\Gamma_t\right\} \to \infty \ \text{as}\ r\to\infty,\\[1ex]
\inf\left\{\sup\{d_\Omega(y,\Gamma_t) : y\in\Omega_t^-,\, d_\Omega(x,y)\leq r\} : t\in\R, x\in\Gamma_t\right\} \to \infty \ \text{as}\ r\to\infty.
\end{cases}
\end{equation}
Moreover, the sets $\Gamma_t$ are assumed to be included in a finite number of graphs. Namely, there is an integer $n\geq1$ such that for each $t\in\mathbb R$, there exist $n$ open subsets $\omega_{i,t}\subset\R$ ($1\leq i\leq n$), $n$ continuous maps $\psi_{i,t}:\omega_{i,t}\to\R$ and $n$ rotations $R_{i,t}$ of $\R^2$, with
\begin{equation}\label{subset3}
\Gamma_t\subset\bigcup\limits_{1\leq i\leq n}R_{i,t}\Bigl(\{x\in\R^2 : x_1 \in \omega_{i,t},\, x_2 = \psi_{i,t}(x_1)\}\Bigr).
\end{equation}

\begin{defi}[\cite{BH1,BH2}]\label{da}
We call $u(t,x)$ a transition front connecting $\mathbf0$ and $\mathbf1$ of \eqref{1.1} if $u(t, x)$ is a classical solution of \eqref{1.1} and there exist sets $\{\Omega_t^\pm\}_{t\in\mathbb R}$ and $\{\Gamma_t\}_{t\in\mathbb R}$ satisfying \eqref{subset1}, \eqref{subset2} and \eqref{subset3} such that for any $\varepsilon>0$, there is a constant $S>0$ satisfying
\begin{equation*}
\begin{cases}
\forall\ (t,x)\in\mathbb R\times\Omega_t^+, \quad (d_\Omega(x,\Gamma_t)\geq S) \ \Rightarrow \ (u(t, x)\geq (1-\varepsilon)\cdot\mathbf1),\\
\forall\ (t,x)\in\mathbb R\times\Omega_t^-, \quad (d_\Omega(x,\Gamma_t)\geq S) \ \Rightarrow \ (u(t,x)\leq \varepsilon\cdot\mathbf1).
\end{cases}
\end{equation*}
Moreover, $u$ is said to have a global mean speed $\gamma (\geq0)$ if
\begin{equation*}
\frac{d_\Omega(\Gamma_t,\Gamma_s)}{|t-s|}\to\gamma \quad \text{as } |t-s|\to\infty.
\end{equation*}
\end{defi}

Transition fronts provide a broad framework for characterizing propagation behavior, which has motivated extensive studies on their qualitative properties in various reaction-diffusion equations. In their seminal work \cite{BH2}, Berestycki and Hamel analyzed the fundamental properties of transition fronts and introduced new types of transition fronts for certain 
time-dependent reaction-diffusion equations. 
For the bistable reaction-diffusion equation, Hamel \cite{H1} proved the existence and uniqueness of the global mean speed in $\mathbb{R}^N$ ($N\geq1$) and showed that this speed is independent of the shape of the fronts. He further established the existence of nonstandard transition fronts. By providing conditions (A1)-(A4), Sheng and Wang \cite{SW18} constructed a new transition front for system \eqref{1.1} that behaves like three planar fronts as $t \to -\infty$ and converges to a V-shaped traveling front as $t \to \infty$. They also proved that any transition front connecting $\mathbf0$ and $\mathbf1$ in $\R\times\R^N$ possesses a unique global mean speed coinciding with the planar wave speed $c$. For further results on transition fronts in the whole space, see \cite{BH1,BH2,bgw,d,hr,hr2,shen,sg,z1,z} and references therein.

The generality of transition fronts is also reflected in the spatial domains. 
In fact, the existence of transition fronts has been established in complex domains such as exterior domains, cylindrical domains and funnel-shaped domains, see \cite{BHM2009,GFL,Guomonobe,hz,LiLinlin} and references therein. 
Here we primarily focus on some studies in exterior domains. 
The interaction between a planar traveling front and an obstacle was first investigated by Berestycki, Hamel and Matano \cite{BHM2009}. 
For the scalar reaction-diffusion equation 
\begin{align*}
\begin{cases}
u_t=\Delta u +f(u), & t\in\R,~x\in\Omega=\R^N\setminus K,~(N\geq2),\\
u_\nu =0 , & t\in\R,~x\in\partial\Omega=\partial K
\end{cases}
\end{align*}
with a bistable nonlinearity $f$
($\exists\theta\in(0,1)$ such that $f(0)=f(\theta)=f(1)=0$, $f'(0)<0$, $f'(1)<0$,
$f<0$ on $(0, \theta)$, $f>0$ on $(\theta, 1)$, $\int_{0}^{1} f(s)\,d s>0$), they proved the existence of an entire solution $u(t,x)$ emanating from a planar traveling front in exterior domains $\Omega$. 
By providing the complete propagation condition (that is, $u(t,x)\to1$ locally uniformly in $x\in\overline\Omega$ as $t\to+\infty$), it was also proved that the $u(t,x)$ will recover to the same planar traveling front after passing the obstacle $K$, and that it is a transition front connecting $0$ and $1$ in $\R\times\overline\Omega$. 
They also indicated that the complete propagation condition holds when the obstacle is star-shaped\footnote{ 
The obstacle $K$ is called star-shaped if either $K=\emptyset$
or there is $x\in\operatorname{Int}(K)$ (the interior of $K$) such that $x+t(y-x) \in \operatorname{Int}(K)$ for all $y
\in \partial K$ and $t \in[0,1)$.} or directionally convex with respect to some hyperplane\footnote{ 
The obstacle $K$ is called directionally convex with respect to a hyperplane $H=\left\{x \in \mathbb{R}^{N}: x \cdot e=a\right\}$,
with $e \in \mathbb{S}^{N-1}$ and $a \in \mathbb{R}$, if for every line $\Sigma$ parallel to $e$, the set $K \cap \Sigma$ is
either a single line segment or empty and if $K \cap H$ is equal to the orthogonal projection of $K$ onto $H$.}. 
Later, Guo and Monobe \cite{Guomonobe} showed that \eqref{1.1} admits an entire solution
emanating from any transition front in $\R\times\R^N$ connecting $0$ and $1$. 
In addition, they proved that the entire solution emanating from a V-shaped traveling front will recover its V-shaped profile under the complete propagation condition. 
In subsequent works, the above results were extend to combustion reaction-diffusion equations, monotone bistable reaction-diffusion systems and  time-periodic bistable reaction-diffusion equations, see \cite{jwz,LiLinlin,slww,yss,ys} and references therein.

In this paper, we focus on the interaction between V-shaped traveling fronts and an obstacle $K$ for monotone bistable reaction-diffusion systems \eqref{1.1}. 
Under assumptions (A1)-(A4), we employ the method of sub- and supersolutions to establish the existence of an entire solution emanating from the V-shaped traveling front. 
By further utilizing the stability of the V-shaped traveling front, we prove that under the complete propagation condition, this entire solution continues to propagate in the form of the same V-shaped profile after passing the obstacle. 
In particular, we prove that the entire solution serves as a transition front connecting $\mathbf0$ and $\mathbf1$  in exterior domains,
 with its global mean speed equal to the planar wave speed $c$. 
Moreover, as shown in \cite{yss}, when the obstacle is star-shaped or directionally convex, the entire solution exhibits complete propagation.

A major difficulty in extending the analysis from the scalar case to the system setting lies in the unknown sign of the nonlinearity $F$ near the stable equilibria $\mathbf0$ and $\mathbf1$. This uncertainty complicates the construction of   sub- and supersolutions. However, under conditions (A2) and (A3), $F$ behaves analogously to a scalar bistable nonlinearity near the stable states $\mathbf0$ and $\mathbf1$ along the Perron-Frobenius directions $R_0$ and $R_1$. This structure enables the construction of sub- and supersolutions by appealing to these Perron-Frobenius directions and the V-shaped traveling front.

We now state the main results of this paper. 
The first theorem establishes the existence and monotonicity of the entire solution emanating from the V-shaped traveling front defined by \eqref{VSHAPE}.

\begin{theo}{\label{t1}}
There exists an entire solution $u(t, x):\R\times\bar\Omega\to(\mathbf0,\mathbf1)$ of \eqref{1.1} that is increasing in time and satisfies
\begin{equation}\label{1.5}
 u(t,x) - V(x_1, x_2 - st) \to \mathbf 0 \quad \text{uniformly in } x  \in \overline{\Omega} \text{ as } t \to -\infty.
\end{equation}
\end{theo}

To prove Theorem \ref{t1}, we require the exponential decay properties of the V-shaped traveling front. 
Since the V-shaped traveling front is a homogeneous transition front connecting $\mathbf0$ and $\mathbf1$ (in the sense of Definition \ref{da} with $\Omega = \mathbb{R}^N$), 
we first establish the time-monotonicity and exponential decay properties of any homogeneous transition front connecting $\mathbf0$ and $\mathbf1$ in Section \ref{s2}. 
These general results then apply directly to the V-shaped traveling front.

Next, we consider the large-time behavior of the entire solution $u(t, x)$ given in Theorem \ref{t1}. 
In fact, for the entire solution arising from the V-shaped traveling front, the V-shaped profile is preserved as $t\to\infty$ under the complete propagation condition.

\begin{theo}{\label{t2}}
Let $u(t,x)$ be the entire solution of \eqref{1.1} given in Theorem \ref{t1}. Assume that $u(t,x)$  propagates completely in the sense of
\begin{equation}\label{1.6}
 u(t,x) \to \mathbf 1 \quad \text{ locally uniformly in } x \in \overline{\Omega} \text{ as } t \to \infty.
\end{equation}
Then
\begin{equation}\label{1.7}
 u(t,x) - V(x_1, x_2 - st) \to \mathbf 0 \quad \text{uniformly in } x \in \overline{\Omega} \text{ as } t \to \infty.
\end{equation}
Furthermore, there holds
\begin{equation}\label{xinfty}
 u(t,x) - V(x_1, x_2 - st) \to \mathbf 0 \quad \text{uniformly in } t\in\R \text{ as } |x| \to \infty.
\end{equation}
\end{theo}

\begin{rem}
By \eqref{p}, \eqref{1.b}, Theorem \ref{t1}, \eqref{1.7}, and \eqref{xinfty}, it follows that the entire solution $u(t,x)$ obtained in Theorem \ref{t2} constitutes a transition front of system \eqref{1.1} in the sense of Definition \ref{da}, with
\[
\Omega_t^\pm = \left\{ x\in\Omega : \pm (x_2 - st - m_*|x_1|) < 0 \right\}, 
\quad
\Gamma_t = \left\{ x\in\Omega : x_2 - m_*|x_1| = st \right\}.
\]
Let $2\alpha \in (0,\pi)$ be the angle of the V-shaped traveling front. From \cite{wangzhicheng2012}, one has
\[
\alpha = \mathrm{arccot}
 m_*
 \ \text{ and }\ \ s = \frac{c}{\sin \alpha}.
\]
For any large $t_1$ and $t_2$, the interfaces $\Gamma_{t_1}$ and $\Gamma_{t_2}$ are far from the obstacle, so that
\[
d_\Omega(\Gamma_{t_1}, \Gamma_{t_2}) = s |t_1 - t_2| \sin \alpha = c |t_1 - t_2|.
\]
Hence, $u(t,x)$ is a transition front of system \eqref{1.1} connecting $\mathbf0$ and $\mathbf1$ in the exterior domain, with a global mean speed equal to the planar wave speed $c$.

\end{rem}

The paper is organized as follows. In Section \ref{s2}, we introduce several auxiliary notations and lemmas that will be used throughout this paper. 
Section \ref{s3} is devoted to establishing the existence and monotonicity of entire solutions emanating from the V-shaped traveling front, that is, we prove Theorem \ref{t1}. 
In Section \ref{s4}, we investigate the large-time behavior of the entire solution originating from the V-shaped traveling front under the assumption of complete propagation, namely, we prove Theorem \ref{t2}.

\section{Prelimilaries}\label{s2}
We first introduce some preliminaries commonly used in this paper.
Let us begin by recalling the notions of sub- and supersolutions.

\begin{defi}\label{dss}
Let $u=(u_1,\dots,u_m)$ be a vector-valued function with $u_i\in C^{0,1}(\R\times\bar\Omega)\cap C^{1,2}(\R\times\Omega)$ for each $i=1,\dots,m$.
If $u$ satisfies
\begin{align*}
\begin{cases}
u_t - D\Delta u - F(u) \leq \mathbf 0 \ ({\rm resp.}\ \geq \mathbf 0) &\ \text{ in }\ \R\times\Omega,\\
\nu\cdot\nabla u \leq \mathbf 0 \ ({\rm resp.}\ \geq \mathbf 0) &\ \text{ on }\ \R\times\partial\Omega,
\end{cases}
\end{align*}
then $u$ is called a subsolution $($resp. supersolution$)$ of \eqref{1.1} in $\R\times\bar\Omega$.
Moreover, if $u$ and $v$ are both subsolutions $($resp. supersolutions$)$ of \eqref{1.1}, then their componentwise maximum $($resp. minimum$)$
$\max(u,v)$ $(\text{resp. } \min(u,v))$
is also a subsolution $($resp. supersolution$)$ of \eqref{1.1}.
\end{defi}

The following comparison principle is standard, see \cite{FT,PM}.
\begin{lem}\label{com}
Let $\underline u, \bar u \in [\mathbf 0, \mathbf 1]$ be a subsolution and a supersolution of \eqref{1.1}, respectively.
If $\underline u(0,x) \leq \bar u(0,x)$ for all $x\in\Omega$, then
\[
\underline u(t,x) \leq \bar u(t,x) \quad \text{for all } (t,x) \in [0,+\infty)\times\bar\Omega.
\]
\end{lem}

Next, we introduce some notations and parameters. 
It follows from (A3) that there exist two sufficiently small positive constants $\eta_0$ and $\eta_1$ such that 
$\eta_0 R_0 \ll R_1$ and $\eta_1 R_1 \ll R_0$.  
Let 
$P_0 := \eta_0 R_0 = (p_1^0, \dots, p_m^0)$,  
$P_1 := R_1 = (p_1^1, \dots, p_m^1)$,
$Q_0 := R_0 = (q_1^0, \dots, q_m^0)$ and 
$Q_1 := \eta_1 R_1 = (q_1^1, \dots, q_m^1)$.
By conditions (A2) and (A3), there exist irreducible constant matrices 
$A_0 = (\mu_{ij}^0)_{m\times m}$ and $A_1 = (\mu_{ij}^1)_{m\times m}$ such that 
$\frac{\partial F_i}{\partial u_j}(\mathbf 0) < \mu_{ij}^0$
and 
$\frac{\partial F_i}{\partial u_j}(\mathbf 1) < \mu_{ij}^1 $
for each 
$ i,j = 1, \dots, m$,
and
$A_0 P_0 \ll -\tfrac{1}{2}\lambda_0 P_0$, 
$A_0 Q_0 \ll -\tfrac{1}{2}\lambda_0 Q_0$, 
$A_1 P_1 \ll -\tfrac{1}{2}\lambda_1 P_1$, 
$A_1 Q_1 \ll -\tfrac{1}{2}\lambda_1 Q_1$,
where the principal eigenvalues of $A_0$ and $A_1$ are negative.

Define a $C^2$ increasing function $\chi(\varsigma)$ satisfying
\[
\chi(\varsigma) \equiv 1 \ \text{for all } \varsigma \ge 1 
\ \text{ and }\ \
\chi(\varsigma) \equiv 0 \ \text{for all } \varsigma \le 0.
\]
We then introduce the positive vector functions
\begin{equation}\label{dpq}
P(\varsigma) = (p_1(\varsigma), \dots, p_m(\varsigma))
\ \text{ and }\ \
Q(\varsigma) = (q_1(\varsigma), \dots, q_m(\varsigma)),
\end{equation}
where the components are defined by
\[
p_i(\varsigma) = \chi(\varsigma)\, p_i^0 + (1-\chi(\varsigma))\, p_i^1
\ \text{ and }\ \
q_i(\varsigma) = \chi(\varsigma)\, q_i^0 + (1-\chi(\varsigma))\, q_i^1,
\]
for each $i = 1, \dots, m$.
It is straightforward to verify that
\begin{align}\label{pq}
& p_i(\varsigma) \in [ p_i^0,p_i^1], \ 
q_i(\varsigma) \in [ q_i^1,q_i^0], \ 
p_i'(\varsigma) \le 0, \  q_i'(\varsigma) \ge 0\quad
\quad \text{for all } \varsigma \in \mathbb R,\\
& p_* := \min_{1 \le i \le m} \inf_{\varsigma \in \mathbb R} p_i(\varsigma) > 0, \ 
q_* := \min_{1 \le i \le m} \inf_{\varsigma \in \mathbb R} q_i(\varsigma) > 0, \notag\\
& p^* := \max_{1 \le i \le m} \sup_{\varsigma \in \mathbb R} p_i(\varsigma) > 0, \ 
q^* := \max_{1 \le i \le m} \sup_{\varsigma \in \mathbb R} q_i(\varsigma) > 0, \notag\\
& 
\max_{1 \le i \le m} \sup_{\varsigma \in \mathbb R} 
\big( |p_i'(\varsigma)| + |p_i''(\varsigma)| + |q_i'(\varsigma)| + |q_i''(\varsigma)| \big) \le M\quad\quad \text{ for some $M>0$}\notag.
\end{align}

For any point $v \in \mathbb R^m$ and any constant $r > 0$, let $B_m(v,r)$ denote the open Euclidean ball with center $v$ and radius $r$ in $\mathbb R^m$.
By the definitions of $\mu_{ij}^0$ and $\mu_{ij}^1$, there exist constants $0 < \varepsilon_0 < \min\left(\frac{p_*}{4}, \frac{q_*}{4}\right)$ and $\varpi > 0$ such that
\begin{align}\label{muij}
& \frac{\partial F_i}{\partial u_j}(u) \le \mu_{ij}^0, \quad 
u \in B_m(\mathbf 0, 4\varepsilon_0) \cap [\mathbf 0, \mathbf 1], \quad i,j=1,\dots,m,  \\
& \frac{\partial F_i}{\partial u_j}(u) \le \mu_{ij}^1, \quad 
u \in B_m(\mathbf 1, 4\varepsilon_0) \cap [\mathbf 0, \mathbf 1], \quad i,j=1,\dots,m, \notag \\
& \sum_{j=1}^{m} \mu_{ij}^0 w_j \le -\varpi w_i, \quad
w \in \big(\mathbb R^m_+ \cap B_m(P_0, 2\varepsilon_0)\big) \cup \big(\mathbb R^m_+ \cap B_m(Q_0, 2\varepsilon_0)\big), \notag \\
& \sum_{j=1}^{m} \mu_{ij}^1 w_j \le -\varpi w_i, \quad
w \in \big(\mathbb R^m_+ \cap B_m(P_1, 2\varepsilon_0)\big) \cup \big(\mathbb R^m_+ \cap B_m(Q_1, 2\varepsilon_0)\big),\notag
\end{align}
where $w = (w_1, \dots, w_m)$.
 Since $F\in C^1([\mathbf0,\mathbf1],\R^m)$, there exists a constant  $\Lambda>0$ such that
\begin{equation}{\label{LA}}
\Lambda=\max\limits_{i=1,\cdots, m}\left(\sup\left\{\sum\limits_{j=1}^{m}\left|\frac{\partial F_i}{\partial u_j}(u)\right|:u\in[\mathbf0,\mathbf1]\right\}\right).
\end{equation}
Denote
\begin{equation}\label{D}
\overline D=\max\limits_{i=1,\cdots, m}D_i>0\ \text{ and }\ \underline D=\min\limits_{i=1,\cdots, m}D_i>0.
\end{equation}

Let us introduce a key lemma established in \cite[Lemma 2.3]{yss}.  
For any point $v \in \mathbb R^2$ and any constant $r > 0$, let $B(v,r)$ denote the open Euclidean ball with center $v$ and radius $r$ in $\mathbb R^2$.
Fix a point $x_0 \in \mathbb{R}^2$ and a constant $R > 0$ such that $B(x_0, R) \subset \Omega$.  
For any $\delta > 0$, let 
$
v_{x_0,R}^\delta = \big(v_{x_0,R}^{\delta,1}, \cdots, v_{x_0,R}^{\delta,m}\big)
$
denote the solution of the Cauchy problem
\begin{align}\label{vr}
\begin{cases}
\begin{aligned}
&(v_{x_0,R}^\delta)_t(t,x) - D \Delta v_{x_0,R}^\delta(t,x) = F(v_{x_0,R}^\delta(t,x)),
&& t > 0,~ x \in \Omega,\\
&\nu \cdot \nabla v_{x_0,R}^\delta(t,x) = \mathbf{0},
&& t > 0,~ x \in \partial\Omega,\\
&v_{x_0,R}^\delta(0,x) =
\begin{cases}
(1-\delta q_*) \cdot \mathbf{1}, & x \in \overline{B(x_0,R)},\\
\mathbf{0}, & x \in \overline{\Omega \backslash B(x_0,R)}.
\end{cases}
\end{aligned}
\end{cases}
\end{align}

\begin{lem}\label{key}
For any small $\delta > 0$, there exist four positive constants $R_1 = R_1(\delta)$, $R_2 = R_2(\delta)$, $R_3 = R_3(\delta)$, and $\hat T = \hat T(\delta)$ satisfying $R_1 < R_2 < R_3$ such that, if $B(x_0,R_3) \subset \Omega$, then
\[
v_{x_0,R_1}^\delta(\hat T, \cdot) \ge (1-\delta q_*) \cdot \mathbf{1} \quad \text{in } \overline{B(x_0,R_2)} (\subset \Omega).
\]
\end{lem}

In what follows, let $U(t,x)=(U_1(t,x),\dots,U_m(t,x))$ be any homogeneous transition front of \eqref{1.1} connecting $\mathbf0$ and $\mathbf1$ in the sense of Definition \ref{da} with $\Omega = \R^2$. Now we turn to show time-monotonicity and  exponential decay properties of $U$.
For any $R>0$ and $\delta\in\left(0,\frac{1}{q^*}\right)$, let 
$
v_{R,\delta}=\left(v_{R,\delta}^1,\cdots,v_{R,\delta}^m\right)
$
be the solution of the Cauchy problem
\begin{equation*}
\begin{cases}
\begin{aligned}
&(v_{R,\delta})_t(t,x)= D\,\Delta v_{R,\delta}(t,x) + F\big(v_{R,\delta}(t,x)\big), 
&& t>0,\ x\in\mathbb{R}^2, \\
&v_{R,\delta}(0,x) =
\begin{cases}
\mathbf1-\delta Q_1, & |x| < R,\\
\mathbf{0}, & |x| \ge R.
\end{cases}
\end{aligned}
\end{cases}
\end{equation*}

\begin{lem}\label{VT1}
For any $T>0$,  $\delta\in\left(0,\frac{1}{q^*}\right)$ and $\vartheta>0$, there exists $R>0$ such that
$$
v_{2R,\delta}(t,x)\geq(1-\vartheta)\cdot\mathbf1-\delta Q_1
\ \text{ for all }t\in[0,T]\text{ and }|x|\leq R.
$$
\end{lem}

\begin{proof}
Let $T>0$ and $\delta\in\left(0,\frac{1}{q^*}\right)$ be any fixed constants.
Let $v_\delta=(v_\delta^1,\cdots,v_\delta^m)$ be the solution of the ordinary differential system
\begin{align}\label{VODE}
\begin{cases}
v_\delta'(t)=F(v_\delta(t)), & t>0,\\[4pt]
v_\delta(0)=\mathbf1-\delta Q_1.
\end{cases}
\end{align}
By \eqref{muij} and $F(\mathbf1)=\mathbf0$, for each $i=1,\cdots,m$,
\begin{align*}
F_i(\mathbf 1-\delta Q_1)=F_i(\mathbf 1-\delta Q_1)-F_i(\mathbf1)
\geq-\sum_{j=1}^m\mu_{ij}^1\delta q_j^1
\geq\delta\varpi q_i\geq0,
\end{align*}
which implies that $\mathbf1-\delta Q_1$ is a subsolution of \eqref{VODE} in $(0,\infty)$,
hence 
$$
v_\delta(t)\geq\mathbf1-\delta Q_1
\ \ \text{ for all }t\geq0.
$$
Define $w=(w_1,\cdots,w_m)$ by $w(t,x)=v_\delta(t)-v_{2R,\delta}(t,x)$.
Then
$$
w(0,x)=
\begin{cases}
\mathbf0,&|x|<2R,\\
\mathbf1-\delta Q_1,&|x|\geq 2R.
\end{cases}
$$
From the proof of formula (3.4) in \cite{SW18}, there exists two constants $B_1$ and $B_2$ such that for each $i=1,\cdots,m$,
\begin{align*}
w_i(t,x)&\leq\frac{B_1}{\sqrt{t}}\int_{\R^2}\exp\left(-B_2\frac{|x-y|^2}{t}\right)w_i(0,y)dy\\
&\leq\frac{B_1}{\sqrt{t}}\int_{|y|\geq2 R}\exp\left(-B_2\frac{|x-y|^2}{t}\right)dy,
\ \ t\in(0,T],\ x\in\R^2.
\end{align*}
Then
\begin{align*}
w_i(t,x)
&\leq{B_1}\int_{|z|\geq \frac{R}{\sqrt t}}\exp\left(-B_2{|z|^2}\right)dz\\
&\leq{B_1}\int_{|z|\geq \frac{R}{\sqrt T}}\exp\left(-B_2{|z|^2}\right)dz,
\  t\in(0,T], \ |x|\leq R. 
\end{align*}
Therefore, for any $\vartheta>0$, there exists $R=R(T,\vartheta)>0$ sufficiently large such that
$w(t,x)\leq\vartheta \cdot\mathbf1$ for all
$ t\in(0,T]$ and $|x|\leq R$.
Together with the definition of $w$, one deduces that the statement of Lemma \ref{VT1} holds.
\end{proof}

The following lemma concerns the large-time behavior of $v_{R,\delta}$. 
Its proof follows similar arguments as in \cite[Lemma 3.2]{SW18} and is therefore omitted.

\begin{lem}\label{VT2}
For any $\delta \in \left(0, \frac{1}{q^*}\right)$, there exist constants $R>0$ and $T>0$ such that
\[
v_{R,\delta}(t,x) \ge \mathbf1 - \delta Q_1
\quad \text{for all } t \ge T \text{ and } |x| \le R.
\]
\end{lem}

Combining Lemma \ref{VT1} and \ref{VT2}, by arguments similar to those in \cite[Proposition 1]{GH}, one can derive that $U(t,x)$ represents an invasion of the state $\mathbf0$ by $\mathbf1$ in the sense of the definition proposed in \cite{BH2}. 
The following lemma is therefore natural.

\begin{lem}\label{invasion}
The following statements are true.
\begin{itemize}
\item[(a)] $U_t(t,x)>0$ for all $(t,x)\in\R\times\R^2$;

\item[(b)] $\lim\limits_{t\to-\infty}U(t,x)\to\mathbf0$ and $\lim\limits_{t\to\infty}U(t,x)\to\mathbf1$ locally uniformly in $x\in\R^2$.
\end{itemize}
\end{lem}

The following lemma focuses on the exponential decay properties of homogeneous transition fronts of system \eqref{1.1} connecting $\mathbf0$ and $\mathbf1$.
	
\begin{lem}\label{tran}
For any point $x_0\in\mathbb R^2$ and any constant $R>0$, there exist  $T=T(R)<0$, $\alpha=\alpha(T)>0$, $\beta=\beta(T)>0$ and $\eta>0$ such that, for each $i=1,\dots,m$,
\begin{equation*}
  U_i(t,x)\le \alpha e^{\eta t}\quad \text{ and }\quad  |\nabla U_i(t,x)| \le \beta e^{\eta t}
\end{equation*}
for all $ t\le T$ and $|x-x_0|<R$.
\end{lem}

\begin{proof}
Assume without loss of generality that $x_0 = (0,0)$. Otherwise, consider the shifted function 
$\widetilde{U}(t,x) = U(t, x - x_0)$. 
To complete the proof, it is necessary to construct a supersolution of the equation satisfied by $U(t,x)$. 
The proof is divided into three steps.

\textit{Step 1: choice of notations and parameters.} 
Choose a constant $\mu$ satisfying
\begin{equation}\label{l1.1}
    0 < \mu \le \frac{c}{3\overline D}.
\end{equation}
Let $h_\mu : [0, \infty) \to \mathbb R$ be a $C^2$ function satisfying
\begin{equation}\label{l1.2}
\begin{cases}
0 \le h_\mu' \le \frac{\sqrt{3}}{3} & \text{on } [0, \infty),\\
h_\mu' = 0 & \text{in a neighborhood of } 0,\\ 
h_\mu(0) > 0 \ \text{ and } \ h_\mu(r) = \frac{\sqrt{3}}{3} r
 &\text{for } r \ge H \text{ with some } H>0,\\
\frac{ h_\mu'(r)}{r} + h_\mu''(r) \le \frac{\mu}{2} & \text{on } [0, \infty).
\end{cases}
\end{equation}
In particular, note that
\begin{equation}\label{h}
\frac{\sqrt{3}}{3} r \le h_\mu(r) \le \frac{\sqrt{3}}{3} r + h_\mu(0) \quad \text{for all } r \ge 0.
\end{equation}
Let $R>0$ be any fixed constant. Define the set
\begin{equation*}
E = \Big\{(t,x) \in \mathbb R \times \mathbb R^2 : t \le 0, \ |x| \le -\Big(\frac{\sqrt 3}{3} c + \frac{\sqrt 3}{2} \mu \overline D\Big) t + R \Big\}.
\end{equation*}
Let $\delta \in (0, \frac{2\varepsilon_0}{q^*})$ be any sufficiently small constant. 
Since $\Phi(\infty) = \mathbf 0$ and $\Phi' \ll \mathbf 0$, there exists $C > \mathcal C$ such that
\begin{equation}\label{l2.a1}
    \Phi(\xi) \le \min\left\{\delta q_* ,\frac{3\delta\mu q_*}{4m\Lambda}\right\} \cdot\mathbf1 \quad \text{for } \xi \ge C,
\end{equation}
where $\mathcal C$ is defined in \eqref{impo1}.

\textit{Step 2: Construction of a supersolution.} 
For all $t \in \mathbb R$ and $x \in \mathbb R^2$, define  
\begin{equation*}
\overline U(t,x) = \min\left\{\Phi(\xi(t,x)) + \delta Q_0, \mathbf 1\right\},
\end{equation*}
where
\begin{equation*}
\xi(t,x) = -h_\mu(|x|) - \Big(\frac{c}{3} - \frac{\mu \overline D}{2}\Big) t + C + h_\mu(0) + \frac{\sqrt{3}}{3} R.
\end{equation*}
We shall prove that $\bar U(t,x)$ is a supersolution of the equation satisfied by $U(t,x)$ in $E$.
Let $i \in \{1, \dots, m\}$ be fixed. It suffices to prove that
\begin{equation*}
\mathcal L_i(\overline U)(t,x) := \overline U_{t}(t,x) - D_i \Delta \overline U(t,x) - F_i(\overline U(t,x)) \ge 0
\end{equation*}
for all $(t,x) \in E$ such that $\bar U_i(t,x)<1$.
By direct calculation, one has
\begin{align*}
\mathcal L_i(\overline U)(t,x)
    =&\ -\Big(\frac{c}{3} + \frac{\mu \overline D}{2}\Big) \Phi_i'(\xi(t,x)) 
       - D_i (h_\mu'(|x|))^2 \Phi_i''(\xi(t,x)) \\
    &\ + D_i \Big( \frac{1}{|x|} h_\mu'(|x|) + h_\mu''(|x|) \Big) \Phi_i'(\xi(t,x)) 
       - F_i(\overline U(t,x)).
\end{align*}

By \eqref{l1.2}, one has $(h_\mu'(|x|))^2 \le 1/3$. 
From \eqref{h}, it follows that
$
\xi(t,x) \ge C 
$
for all $(t,x) \in E$.
Since $\Phi_i' < 0$, by \eqref{impo1} and \eqref{l1.2}, we obtain
\begin{align*}
\mathcal L_i(\overline U)(t,x)
    &\ge -\Big(\frac{c}{3} + \frac{\mu \overline D}{2}\Big)\Phi_i'(\xi(t,x)) 
       - \frac{D_i}{3} \Phi_i''(\xi(t,x)) 
       + \frac{\mu D_i}{2} \Phi_i'(\xi(t,x)) 
       - F_i(\overline U(t,x)) \\
    &\ge \frac{1}{3} F_i(\Phi(\xi(t,x))) - F_i(\overline U(t,x)).
\end{align*}
Using \eqref{pq} and \eqref{l2.a1}, it follows that
\[
\mathbf 0 \le \overline U(t,x) \le \Phi(\xi(t,x)) + \delta q^* \mathbf 1 
\le 2\delta q^* \mathbf 1 \le 4\varepsilon_0 \mathbf 1 \quad \text{for all } (t,x) \in E.
\]
Since $F_i(\mathbf 0) = 0$, by \eqref{dpq}, \eqref{pq}, \eqref{muij}, \eqref{LA} and \eqref{l2.a1}, 
one infers from the mean value theorem that there exist  $\theta_1, \theta_2, \theta_3 \in (0,1)$ such that
\begin{align*}
\mathcal L_i(\overline U)(t,x)
    &\ge \frac{1}{3} \sum_{j=1}^m \frac{\partial F_i}{\partial u_j} \big(\theta_1 \Phi(\xi(t,x)) \big) \Phi_j(\xi(t,x)) 
       - \sum_{j=1}^m \frac{\partial F_i}{\partial u_j} \big(\theta_2 \Phi(\xi(t,x)) \big) \Phi_j(\xi(t,x)) \\
    &\quad - \sum_{j=1}^m \frac{\partial F_i}{\partial u_j} \big(\overline U(t,x) - \theta_3 \delta Q_0 \big) \delta q_j^0 \\
    &\ge -\frac{1}{3} \Lambda \sum_{j=1}^m \Phi_j(\xi(t,x)) - \Lambda \sum_{j=1}^m \Phi_j(\xi(t,x)) - \delta \sum_{j=1}^m \mu_{ij}^0 q_j^0 \\
    &\ge -\frac{4}{3} \Lambda \sum_{j=1}^m \Phi_j(\xi(t,x)) + \delta \varpi q_* \\
    &\ge 0.
\end{align*}

\textit{Step 3: Proof of Lemma \ref{tran}.} 
By Lemma \ref{invasion} and the fact that $\overline U(0,x) > \mathbf 0$ for all $|x|\le R$, there exists $T < 0$ such that
\[
U(T,x) \le \min\left(\overline U(0,x), \Phi(C + h_\mu(0)) + \delta Q_0, \mathbf 1 \right), \quad |x|\le R.
\]
From \cite[Theorem~1.6]{SW18}, $U(t,x)$ has a unique global mean speed equal to the planar wave speed $c$. Together with Lemma \ref{invasion},  by decreasing $T$ is necessary, there holds
\begin{equation*}
U(t,x)\leq \Phi(C+h_\mu(0))+\delta Q_0\ \ \text{ for all } t\leq T \text{ and }|x|\leq R- \frac{\sqrt3}{2}c(t-T).
\end{equation*}
By (\ref{l1.1}), we have
\begin{equation*}
U(t,x)\leq \Phi(C+h_\mu(0))+\delta Q_0\ \ \text{ for all } t\leq T\text{ and }|x|\leq R- \Big(\frac{\sqrt3}{3}c+\frac{\sqrt3}{2}\mu\overline D\Big)(t-T).
\end{equation*}
Combining this with \eqref{dpq} and \eqref{l2.a1}, one gets
\begin{align}\label{IE}
U(t+T,x)\leq \Phi(C+h_\mu(0))+\delta Q_0\leq 2\delta q^*\mathbf 1\leq 4\varepsilon_0\mathbf 1\ \ \text{ for all }(t,x)\in E.
\end{align}
On the other hand, by \eqref{h},
\begin{equation*}
\xi(t,x) \le -\frac{\sqrt3}{3} |x| - \Big(\frac{c}{3} + \frac{\mu \bar D}{2}\Big) t + C + \frac{\sqrt3}{3} R + h_\mu(0) = C + h_\mu(0)
\end{equation*}
for all $(t,x) \in \partial E:=
 \Big\{ (t,x) \in \mathbb R \times \mathbb R^2 : t \le 0, \ |x| = -\Big(\frac{\sqrt3}{3} c + \frac{\sqrt3}{2} \mu \bar D \Big) t + R \Big\}$.
Since $\Phi' \ll \mathbf 0$, it follows that
\begin{align}\label{OPE}
\overline U(t,x) \ge \min\left(\Phi(C + h_\mu(0)) + \delta Q_0, \mathbf 1 \right)\ge U(t+T,x) \quad \text{for all }(t,x) \in \partial E.
\end{align}

Define
\begin{align*}
  \varepsilon^*=\inf\left\{\varepsilon>0:U(t+T,x)-\delta\varepsilon Q_0\leq\overline U (t,x) \text{ for all }(t,x)\in E\right\}.
\end{align*}
We claim that $\varepsilon^*=0$.
Assume by contradiction that $\varepsilon^*>0$.
Then there exists an index $i_0\in\{1,\dots,m\}$ and sequences $(\varepsilon_{n})_{n\in\mathbb N}\subset\mathbb R$ and
 $(t_{n},x_{n})_{n\in\mathbb N}\subset E$ such that
$\varepsilon_{n}\ge\varepsilon^*$ for all $n\in\mathbb N$, 
$\varepsilon_{n}\to\varepsilon^*$ and
\begin{equation}\label{l1.5}
\overline U_{i_0}(t_{n},x_{n}) - U_{i_0}(t_{n}+T,x_{n}) 
+ \delta \varepsilon_{n} q_{i_0}^0 \to 0 
\quad \text{as } n\to\infty.
\end{equation}
If $t_n\to-\infty$, then
$U_{i_0}(t_n,x_n)\to0$ as $n\to\infty$ uniformly in $E$ by Lemma \ref{invasion}. 
This implies that $\overline U_{i_0}(t_{n},x_{n}) 
+ \delta \varepsilon_{n} q_{i_0}^0 \to 0$ 
as $ n\to\infty$, a contradiction with
 $\overline U_{i_0}(t_n,x_n) + \delta \varepsilon_{n} q_{i_0}^0 
\ge \delta \varepsilon^* q_{i_0}^0$ for all $n\in\mathbb N$.
Hence there exists $(t^*,x^*)\in E$ such that $(t_n,x_n)\to(t^*,x^*)$ as $n\to\infty$.
By \eqref{l1.5}, one has
\begin{equation}\label{txx}
\overline U_{i_0}(t^*,x^*) - U_{i_0}(t^*+T,x^*) 
+ \delta \varepsilon^* q_{i_0}^0 = 0.
\end{equation}

Set the function $Z = (Z_1, \dots, Z_m)$ by
$Z(t,x) = \overline U(t,x) - U(t+T,x) + \delta \varepsilon^* Q_0$.
From \eqref{muij} and \eqref{IE}, for each $i = 1, \dots, m$, we have
\begin{align*}
&(U_i(t+T,x) - \delta \varepsilon^* q_i^0)_t
- D_i \Delta (U_i(t+T,x) - \delta \varepsilon^* q_i^0)
- F_i(U(t+T,x) - \delta \varepsilon^* Q_0)\\
={}& F_i(U(t+T,x)) - F_i(U(t+T,x) - \delta \varepsilon^* Q_0)\\
\le{}& \delta \varepsilon^* 
\sum_{j=1}^m \mu_{ij}^0 q_j^0
\le -\delta \varepsilon^* \varpi q_i^0
\le 0.
\end{align*}
Together with Step 2 and (A4), there exists an $m$-order matrix function
$B(t,x) = (b_{ij}(t,x))_{1 \le i,j \le m}$, 
where each entry $b_{ij}(t,x)$ is bounded and satisfies $b_{ij}(t,x) \ge 0$ for $i \ne j$, 
such that
\[
Z_t(t,x) - D\Delta Z(t,x) + B(t,x)Z(t,x) \ge \mathbf 0
\quad \text{for all } (t,x)\in E.
\]
By the definition of $\varepsilon^*$, together with \eqref{OPE} and \eqref{txx}, one obtains 
$Z_{i_0}(t^*,x^*) = 0$ and $Z_{i_0}(t,x) \ge 0$ for all $(t,x)\in\overline E$.
Then, by the maximum principle (see \cite{PM}), it follows that $Z_{i_0}(t,x) \equiv 0$ in $E$.
However, by \eqref{OPE}, we have
\[
\overline U_{i_0}(t,x)
= U_{i_0}(t+T,x) - \delta \varepsilon^* q_{i_0}^0
< \overline U_{i_0}(t,x)
\quad \text{for } (t,x)\in\partial E,
\]
which leads to a contradiction.

As a conclusion, $\varepsilon^* = 0$.
It follows that
$U(t+T,x) \le \overline U(t,x) \le \Phi(\xi(t,x)) + \delta Q_0$
for all $(t,x)\in E$.
Since $\delta$ was arbitrary, we obtain
\[
U(t+T,x) \le \overline U(t,x) \le \Phi(\xi(t,x))
\quad \text{for all } (t,x)\in E.
\]
By \eqref{h}, there holds
\[
\xi(t,x) \ge -\Big(\frac{c}{3} + \frac{\mu\overline D}{2}\Big)t + C \ge 0,
\quad  t \le 0,\ |x| < R.
\]
From \eqref{1.a}, it follows that for each $i = 1, \dots, m$,
\[
U_i(t,x) \le \Phi_i(\xi(t-T,x))
\le a e^{-b\xi(t-T,x)}
\le a e^{b\left(\frac{c}{3} + \frac{\mu\overline D}{2}\right)(t-T)},
\quad t \le T,\ |x| < R.
\]
By standard parabolic interior estimates, there exists a constant $a_0 > 0$ such that for each $i = 1, \dots, m$,
\[
|\nabla U_i(t,x)|
\le a_0 e^{b\left(\frac{c}{3} + \frac{\mu\overline D}{2}\right)(t-T)},
\quad t \le T,\ |x| < R.
\]
By setting 
$\alpha = a e^{-b\left(\frac{c}{3} + \frac{\mu\overline D}{2}\right)T}$, 
$\beta = a_0 e^{-b\left(\frac{c}{3} + \frac{\mu\overline D}{2}\right)T} $
and
$\eta = b\left(\frac{c}{3} + \frac{\mu\overline D}{2}\right)$,
we finish the proof of Lemma \ref{tran}.
\end{proof}

\section{Entire solution emanating from the V-shaped traveling front}\label{s3}
This section is devoted to establishing the existence and monotonicity of the entire solution emanating from the V-shaped traveling front, that is, we prove Theorem \ref{t1}.

Since $K$ is a compact subset of $\mathbb{R}^2$, there exists a constant $L > 0$ such that
\begin{align*}
|x| < L \quad \text{for all } x \in \overline K.
\end{align*}
Note that the V-shaped traveling front is a homogeneous transition front of \eqref{1.1} connecting $\mathbf 0$ and $\mathbf 1$. 
By Lemma \ref{tran}, there exist constants $T = T(L) < 0$, $\alpha = \alpha(T) > 0$, $\beta = \beta(T) > 0$ and $\eta > 0$ such that, for each $i = 1, \dots, m$,
\begin{equation}\label{TL}
V_i(x_1, x_2 - s(t + 1)) \le \alpha e^{\eta t}
\quad \text{and} \quad
|\nabla V_i(x_1, x_2 - s(t + 1))| \le \beta e^{\eta t}
\end{equation}
for all $t \le T$ and $x \in \overline K$. 
Moreover, from the proof of Lemma \ref{tran}, the constant $\eta$ can be chosen independently of $L$.

Let $\tilde{\zeta} : \overline{\Omega} \to \mathbb{R}$ be a nonnegative $C^2$ function with compact support in $\overline{\Omega}$ satisfying $\nu \cdot \nabla \tilde{\zeta} = 1$ on $\partial \Omega$. 
Then, the functions $\tilde{\zeta}$, $\tilde{\zeta}_{x_i}$ ($i = 1, 2$) and $\Delta \tilde{\zeta}$ are bounded and compactly supported. 
Such a truncated function $\tilde{\zeta}$ can be constructed by applying the classical distance function near the boundary $\partial \Omega$, as described in \cite{GT}.
Consequently, there exists a constant $C_1 > 0$ such that
\[
\zeta(x) = \tilde{\zeta}(x) + C_1 > 1 
\quad \text{for all } x \in \overline{\Omega}
\]
and
\begin{equation}\label{zeta}
\left\| \frac{\Delta \zeta}{\zeta} \right\|_{L^{\infty}(\overline{\Omega})}
\le \min\left\{\frac{\eta}{\overline D},\frac{\varpi}{2\bar D}\right\}.
\end{equation}

We now proceed to construct a pair of sub- and supersolutions of \eqref{1.1} based on the V-shaped traveling front $V(x_1, x_2 - st)$. 
According to  \cite{NT}, there exists a unique convex function $\psi(x)$ satisfying
\begin{equation}\label{2C.8}
  s = \frac{\psi''(x)}{1 + (\psi'(x))^2} + c \sqrt{1 + (\psi'(x))^2}, 
  \quad x \in \mathbb{R}
\end{equation}
together with the bounds
\begin{align}\label{psii}
|\psi'(x)| \le m_* 
\quad \text{and} \quad
m_* |x| \le \psi(x) \le m_* |x| + m_0
\end{align}
for some constant $m_0 > 0$, where $m_*$ is defined in \eqref{mstar}.

\begin{lem}\label{l2.2}
There exist constants $w > 0$ and $\bar T < 0$ such that the functions 
${V}^{\pm} = (V_1^{\pm}, \dots, V_m^{\pm})$ defined by
\begin{equation*}
  V^{-}(t,x) = 
  \max\left\{
    V(x_1, z^-) - \frac{2\beta}{q_*} Q(\xi^-(t,x)) \zeta(x) e^{\eta t},
    \mathbf{0}
  \right\}
\end{equation*}
and
\begin{equation*}
  V^{+}(t,x) =
  \min\left\{
    V(x_1, z^+) + \frac{2\beta}{q_*} Q(\xi^+(t,x)) \zeta(x) e^{\eta t},
    \mathbf{1}
  \right\}
\end{equation*}
are sub- and supersolutions of \eqref{1.1} for all 
$t \le \bar T$ and $x  \in \overline{\Omega}$ respectively,
where $\beta>0$ is given by \eqref{TL}, the function $Q(\cdot)$ is defined in \eqref{dpq} and
\[
\xi^{\pm}(t,x) = z^{\pm} - \psi(x_1), 
\qquad 
z^{\pm} = x_2 - s t \mp w e^{\eta t}.
\]
\end{lem}

\begin{proof}
We only prove that $V^+$ is a supersolution of \eqref{1.1}, the proof for the subsolution $V^-$ is similar. 
We first introduce some auxiliary parameters.
Let $\vp_0$, $q^*$, $q_*$, $M$, $\Lambda$ and $\bar D$ be the constants defined in Section \ref{s2}.  Take any
\begin{equation}\label{dl}
0 < \delta \le\frac{2\varepsilon_0}{q_*}.
\end{equation} 
By \eqref{1.b} and the limits $\Phi(-\infty) = \mathbf{1}$, $\Phi(\infty) = \mathbf{0}$, we have  
$V(y,z) \to \mathbf{1}$ as $z - m_* |y| \to -\infty$ and  
$V(y,z) \to \mathbf{0}$ as $z - m_* |y| \to \infty$.  
Hence there exists a constant $C > 1$ such that
\begin{equation}\label{2l.3}
\begin{cases}
V(y,z) \ge (1 - \delta q_*) \cdot\mathbf{1} & \text{for all } z - m_* |y| \le -C,\\[2mm]
V(y,z) \le \delta q_* \cdot\mathbf{1} & \text{for all } z - m_* |y| \ge C.
\end{cases}
\end{equation}
By the properties of $V$, there exists $\kappa > 0$ satisfying
\begin{equation}\label{kap}
\max_{ i =1,\cdots, m}
\left\{
\sup_{-C \le z - m_* |y| \le C} (V_i)_z(y,z)
\right\}
= -\kappa < 0.
\end{equation}
Choose $w > 1$ sufficiently large such that
\begin{align}\label{om}
\omega \eta \kappa
\ge{}&
\frac{2\beta}{q_*}
\left(
 (s + 1)M +q^* \Lambda
 + \bar D M \left( \frac{s^2}{c^2} + \|\psi''\|_{L^{\infty}(\mathbb{R})} \right)
\right) \|\zeta\|_{L^{\infty}(\overline{\Omega})} \\
&+ \frac{2\beta}{q_*} q^* \bar D \|\Delta \zeta\|_{L^{\infty}(\overline{\Omega})}
+ \frac{4\beta s}{c q_*} \bar D M \|\nabla \zeta\|_{L^{\infty}(\overline{\Omega})}. \notag
\end{align}  
Take $\overline{T} \le T = T(L)$ sufficiently small such that
\begin{equation}\label{BART}
\omega e^{\eta \bar T} \le \min\left\{1, \frac{1}{\eta}\right\}
\quad \text{and} \quad
-L - m_* L - s \overline{T} - 1 \ge 1.
\end{equation}

Fix an index $i \in \{1, \dots, m\}$.  
We now verify the boundary condition.  
By \eqref{psii} and \eqref{BART}, it follows that
\[
\xi^+(t,x) 
\ge x_2 - m_* |x_1| - s \bar T - 1
\ge -L - m_* L - s \bar T - 1
\ge 1,
\quad t \le \bar T,\ x \in \partial \Omega.
\]
Hence, by \eqref{dpq}, we have  
$Q(\xi^+(t,x)) = Q_0$
for all $t \le \bar T$ and $x \in \partial \Omega$.  
Moreover, using  $ \nu\cdot \nabla \zeta=1$, \eqref{TL} and \eqref{BART}, we obtain
\[
\nu \cdot\nabla V_i^+(t,x) 
=\nu \cdot \nabla V_i(x_1, z^+) 
+ \frac{2\beta}{q_*} q_i^0 e^{\eta t}
\ge -\beta e^{\eta t} + 2\beta e^{\eta t}
\ge 0
\]
for all $(t,x) \in (-\infty, \bar T] \times \partial \Omega$ such that $V_i^+(t,x) < 1$.

It suffices to show that 
\[
\mathcal L_i(V^+)(t,x)
:= (V_i^+)_t(t,x) - D_i \Delta V_i^+(t,x) - F_i(V^+(t,x))
\ge 0
\]
for all $t \le \overline T$ and $x \in \overline \Omega$ such that $V_i^+(t,x) < 1$.
By direct calculations, we obtain
\begin{align*}
\mathcal L_i(V^+)(t,x)
={}& - \omega \eta e^{\eta t}(V_i)_{x_2}(x_1, z^+)
+ \frac{2\beta}{q_*} q_i'(\xi^+) (-s - \omega \eta e^{\eta t}) \zeta(x) e^{\eta t} \\
&+ \frac{2\beta}{q_*} \eta q_i(\xi^+) \zeta(x) e^{\eta t}
- D_i \frac{2\beta}{q_*} \Delta q_i(\xi^+) \zeta(x) e^{\eta t}\\
&- D_i \frac{2\beta}{q_*} q_i(\xi^+) \Delta \zeta(x) e^{\eta t}
- D_i \frac{4\beta}{q_*} \nabla q_i(\xi^+) \cdot \nabla \zeta(x) e^{\eta t}\\
&+ F_i(V(x_1,z^+)) - F_i(V^+(t,x)).
\end{align*}

If $\xi^+ \le -C- m_0$, then $z^+ - m_*|x_1| \le -C$ and $Q(\xi^+) = Q_1$ by \eqref{psii} and \eqref{dpq}.
Using \eqref{dl} and \eqref{2l.3}, we have
\[
(1 - 4\varepsilon_0)\cdot\mathbf 1
\le (1 - \delta q_*)\cdot\mathbf 1
\le V^+(t,x)
\le V(x_1,z^+)
\le \mathbf 1.
\]
By \eqref{muij}, it follows that
\begin{align*}
F_i(V^+(t,x)) - F_i(V(x_1,z^+))
&\le \frac{2\beta}{q_*} \zeta(x) e^{\eta t} \sum_{j=1}^m \mu_{ij}^1 q_j^1
\le -\frac{2\beta}{q_*} \varpi q_i^1 \zeta(x) e^{\eta t}
\le 0. \label{ineqFi}
\end{align*}
Since $V_z(y,z) \ll \mathbf 0$ for all $(y,z) \in \R^2$, by \eqref{D} and \eqref{zeta} we obtain
\begin{align*}
\mathcal L_i(V^+)(t,x)
\ge  \frac{2\beta}{q_*} q_i^1 \zeta(x) e^{\eta t}
\left( \eta - \overline D \Big\|\frac{\Delta \zeta}{\zeta}\Big\|_{L^\infty(\overline\Omega)} \right)
\ge0.
\end{align*}

If $\xi^+ \ge C$, then $z^+ - m_*|x_1| \ge C$ and $Q(\xi^+) = Q_0$ by \eqref{dpq} and \eqref{psii}.
Using \eqref{dl} and \eqref{2l.3}, we obtain
\[
\mathbf 0 \le V^+(t,x)
\le V(x_1,z^+)
\le \delta q_* \cdot\mathbf 1
\le 4\varepsilon_0\cdot \mathbf 1.
\]
By \eqref{muij}, it follows that
\begin{align*}
F_i(V^+(t,x)) - F_i(V(x_1,z^+))
&\le \frac{2\beta}{q_*} \zeta(x) e^{\eta t} \sum_{j=1}^m \mu_{ij}^0 q_j^0
\le -\frac{2\beta}{q_*} \varpi q_i^0 \zeta(x) e^{\eta t}
\le 0. \label{ineqFi}
\end{align*}
A similar argument yields $\mathcal L_i(V^+)(t,x) \ge 0$.

If $-C - m_0 \le \xi^+ \le C$, then $-C \le z^+ - m_*|x_1| \le C$ by \eqref{psii}.
It follows from \eqref{mstar}, \eqref{pq} and \eqref{psii} that
\begin{equation}\label{naq}
|\nabla q_i(\xi^+)| = |q_i'(\xi^+)| \sqrt{1 + (\psi'(x_1))^2}
\le |q_i'(\xi^+)| \sqrt{1 + m_*^2}
\le \frac{sM}{c}
\end{equation}
and
\begin{equation}\label{deq}
|\Delta q_i(\xi^+)| = |q_i''(\xi^+)(1 + (\psi'(x_1))^2) -q_i'(\xi^+) \psi''(x_1)|
\le \left(\frac{s^2}{c^2} + \|\psi''\|_{L^\infty(\mathbb R)}\right) M.
\end{equation}
By \eqref{pq}  and \eqref{LA}, we obtain
\begin{align*}
F_i(V^+(t,x)) - F_i(V(x_1,z^+))
\le \frac{2\beta}{q_*} \Lambda q_i(\xi^+) \zeta(x) e^{\eta t} 
\le \frac{2\beta}{q_*} q^* \Lambda \|\zeta\|_{L^\infty(\overline\Omega)} e^{\eta t}.
\end{align*}
Combining \eqref{pq}, \eqref{kap}, \eqref{om} and \eqref{BART}, we derive
\begin{align*}
\mathcal L_i(V^+)(t,x)
\ge {}& \omega \eta \kappa e^{\eta t}
- \frac{2\beta}{q_*} (s + 1) M \|\zeta\|_{L^\infty(\overline\Omega)} e^{\eta t}
- \frac{2\beta}{q_*} q^* \Lambda \|\zeta\|_{L^\infty(\overline\Omega)} e^{\eta t}\\
&- \frac{2\beta}{q_*} \overline D M
\left( \frac{s^2}{c^2} + \|\psi''\|_{L^\infty(\mathbb R)} \right)
\|\zeta\|_{L^\infty(\overline\Omega)} e^{\eta t}\\
&- \frac{2\beta}{q_*} q^* \overline D \|\Delta\zeta\|_{L^\infty(\overline\Omega)} e^{\eta t}
- \frac{4\beta s}{cq_*} \overline D M \|\nabla\zeta\|_{L^\infty(\overline\Omega)} e^{\eta t}\\
\ge {}& 0.
\end{align*}

Hence, the proof is complete.
\end{proof}

\begin{proof}[Proof of Theorem \ref{t1}]
For any $n \in \mathbb{N} \cap [1 + |\bar T|, \infty)$, let $u_n(t,x)$ denote the solution of \eqref{1.1} for all $t \ge -n$ and $x   \in \overline{\Omega}$ with the initial condition $u_n(-n, x) = V^+(-n, x)$. 
It follows from Lemma \ref{l2.2} that 
$\mathbf 0 \le V^-(-n, x) \le u_n(-n, x) \le \mathbf 1$ for all  $x \in \overline{\Omega}$.
By the comparison principle, we have
\begin{equation}\label{Vu}
\mathbf 0 \le V^-(t, x) \le u_n(t, x) \le V^+(t, x) \le \mathbf 1,
\quad t \in [-n, \bar T],\ x \in \overline{\Omega}.
\end{equation}
Hence, $\mathbf 0 \le u_n(-n+1, \cdot) \le V^+(-n+1, \cdot) = u_{n-1}(-n+1, \cdot) \le \mathbf 1$ in $\overline{\Omega}$. 
Applying the comparison principle again yields $u_n(t, x) \le u_{n-1}(t, x)$ for all $t \in [-n+1, \infty)$ and $x \in \overline{\Omega}$, it implies that $u_n$ is decreasing in $n$.
By standard parabolic estimates, there exists an entire solution $u^*$ of \eqref{1.1} defined in $\mathbb{R} \times \overline{\Omega}$ such that 
\[
u_n(t, x) \to u^*(t, x) 
\quad \text{uniformly in } (t, x) \in \mathbb{R} \times \overline{\Omega}
\text{ as } n \to \infty.
\]
From \eqref{Vu}, we obtain
\[
V^-(t, x) \le u^*(t, x) \le V^+(t, x)
\quad \text{for all } t \le \bar T \text{ and } x \in \overline{\Omega}.
\]
By the definition of $V^\pm$, it follows that
\[
u^*(t, x) \to V(x_1, x_2 - s t)
\quad \text{uniformly in } x \in \overline{\Omega} \text{ as } t \to -\infty.
\]

Since $V_z(y, z) \ll \mathbf 0$ for all $(y, z) \in \mathbb{R}^2$, we have $V^+(t, x)$ is increasing for sufficiently negative $t$, it implies that $(u_n)_t(-n, x) \gg \mathbf 0$ for large $n$. 
Consider the following  system
\begin{equation*}
\begin{cases}
(u_n)_{tt} - D \Delta (u_n)_t - F'(u_n)(u_n)_t = \mathbf 0, & (t, x) \in (-n, \infty) \times \Omega,\\
\nu \cdot \nabla (u_n)_t = 0, & (t, x) \in (-n, \infty) \times \partial \Omega,\\
(u_n)_t(-n, x) \gg \mathbf 0, & x \in \Omega.
\end{cases}
\end{equation*}
By the strong maximum principle and the Hopf boundary lemma, it follows that 
$(u_n)_t(t, x) \gg \mathbf 0$ for all $t \in [-n, \infty)$ and $x \in \overline{\Omega}$. 
Passing to the limit as $n \to \infty$, we have $u^*_t(t, x) \ge \mathbf 0$ for all $(t, x) \in \mathbb{R} \times \overline{\Omega}$. 
Applying again the strong maximum principle and the Hopf boundary lemma gives
\begin{align*}
u^*_t(t, x) \gg \mathbf 0
\quad \text{for all } (t, x) \in \mathbb{R} \times \overline{\Omega}.
\end{align*}
Furthermore, by the same arguments, we also obtain
$\mathbf 0 \ll u^*(t, x) \ll \mathbf 1$
for all $ (t, x) \in \mathbb{R} \times \overline{\Omega}$.
This completes the proof of Theorem \ref{t1}.
\end{proof}

\section{Large-time behavior of entire solution}\label{s4}

This section studies the large-time behavior of the entire solution given by Theorem \ref{t1} under the complete propagation condition, by employing the sub- and supersolution method and the stability of the V-shaped traveling front.

\subsection{Sub- and supersolutions}
In this subsection, we construct several sub- and supersolutions in preparation for the analysis.
Let $\varepsilon_0$, $q_*$, $q^*$, $\varpi$, $M$, $\bar{D}$, $\Lambda$ be the positive constants defined as in Section \ref{s2}, and let $P$, $Q$, $\zeta$ and $\psi$ be the functions defined in \eqref{dpq}, \eqref{zeta} and \eqref{psii}, respectively.
Since the obstacle $K$ is compact, there exists $C_K>0$ such that
\begin{align}\label{ck}
x_2-m_*|x_1|\leq C_K\quad\text{ for all }x\in\partial\Omega. 
\end{align}

It follows from \cite{wangzhicheng2012} that there exist constants $\tau_1$ and $\tau_2$ such that
\begin{equation*}
 \Phi\Big(\frac{c}{s}(z - m_* |y| + \tau_1)\Big)
 \le V(y, z)
 \le \Phi\Big(\frac{c}{s}(z - m_* |y| + \tau_2)\Big)
\end{equation*}
for all  $(y,z) \in \mathbb{R}^2$. 
Therefore, in the same way as \cite[Lemma 3.1]{Guomonobe}, there exist positive constants $d_1$, $d_2$, $\tilde\lambda_1$ and $\tilde\lambda_2$ such that,
for each $i = 1, \dots, m$,
\begin{equation}\label{V22}
\begin{cases}
\|\nabla V_i(y, z)\|_{L^{\infty}(\mathbb{R}^2)} \le d_1 e^{\tilde\lambda_1 (z - m_* |y|)}&\text{for all }(y,z)\in\R^2\text{ such that }
z - m_* |y| \le 0,\\
\|\nabla V_i(y, z)\|_{L^{\infty}(\mathbb{R}^2)} \le d_2 e^{-\tilde\lambda_2 (z - m_* |y|)}
&\text{for all }(y,z)\in\R^2\text{ such that } 
z - m_* |y| > 0.
\end{cases}
\end{equation}

\begin{lem}\label{sub1}
Let $H$ be any fixed constant. Then for any 
small $\delta>0$,
there exist constants $\sigma>0$, $\rho>0$ and $T>0$ ($\sigma$ and $\rho$ are independent of $\delta$) such that the function
\[
\underline{w}_1(t,x) = \max\left\{ w_1(t,x), \mathbf 0 \right\}
\]
is a subsolution of \eqref{1.1} for all $t \ge 0$ and $x \in \overline{\Omega}$, 
where the function $w_1 = (w_1^1, \dots, w_1^m)$ is defined by
\[
w_1(t,x) = V(x_1, z) - \delta Q(\xi_1(t,x)) \zeta(x) e^{-\sigma t}
\]
with
\[
\xi_1(t,x) = z - \psi(x_1), \quad 
z = x_2 - s(t+T) + \rho \delta (1 - e^{-\sigma t}) + H.
\]
\end{lem}

\begin{proof}
Let us first introduce some parameters and notations.
Take any 
\begin{align}\label{dll}
0<\delta<\min\left\{\frac{2\varepsilon_0}{q^*\|\zeta\|_{L^{\infty}(\overline\Omega)}},1\right\}\left(<\frac{2\varepsilon_0}{q^*}\right).
\end{align}
By the properties of the V-shaped traveling front, there exists a constant $C>1$ such that
\begin{equation}\label{VC}
\begin{cases}
V(y,z)\geq (1-\delta q_*)\cdot \mathbf1 &\text{for all } z - m_*|y| \le -C,\\
V(y,z)\leq \delta q_*\cdot \mathbf1 &\text{for all } z - m_*|y| \ge C.
\end{cases}
\end{equation}
Furthermore, there exists $\kappa>0$ such that
\begin{equation}\label{kapppp}
\max_{i=1,\cdots,m} \left\{ \sup_{-C \le z - m_*|y| \le C} (V_i)_z(y,z) \right\} = -\kappa < 0.
\end{equation}
Set
\begin{equation}\label{si}
\sigma := \min\left\{ s \tilde\lambda_1, \frac{\varpi}{2} \right\}.
\end{equation}
Take $\rho>0$ sufficiently large such that
\begin{align}\label{rho}
\sigma \rho \kappa 
\geq{}
&\left( \sigma q^*+s M+ \Lambda q^*+ 
\bar D  \left(\frac{s^2}{c^2} + \|\psi''\|_{L^\infty(\mathbb R)}\right) M\right)
\|\zeta\|_{L^\infty(\bar\Omega)} \\
&+ \bar D  q^*\|\Delta \zeta\|_{L^\infty(\bar\Omega)} 
+ \frac{2}{c}s\bar DM \|\nabla\zeta\|_{L^\infty(\bar\Omega)} \notag.
\end{align}
Choose $T>0$ large enough such that
\begin{equation}\label{T}
C_K - sT + \rho + H \le 0\ \text{ and }\
d_1 e^{\tilde\lambda_1(C_K - sT + \rho + M)} \le \delta q_*.
\end{equation}

Let $i \in \{1, \dots, m\}$ be any fixed index. We now verify the boundary condition. 
From \eqref{psii}, \eqref{ck}, \eqref{dll} and \eqref{T}, one has $\xi_1(t,x) \le C_K - sT + \rho + H \le 0$ for all $t \ge 0$ and $x \in \partial \Omega$, 
hence $Q(\xi_1(t,x)) = Q_1$ by \eqref{dpq}.
It follows from \eqref{ck}, \eqref{V22}, \eqref{dll}, \eqref{si} and \eqref{T} that
\begin{align*}
\nu\cdot\nabla V_i(x_1, z)   \le \|\nabla V_i(x_1, z)\|_{L^{\infty}(\mathbb R^2)}
&\le d_1 e^{\tilde\lambda_1(x_2 - s(t+T) + \rho \delta (1 - e^{-\sigma t}) + H - m_* |x_1|)}\\
&\le d_1 e^{-s \tilde\lambda_1 t} e^{\tilde\lambda_1(C_K- sT + \rho + H )}\\
&\le \delta q_* e^{-\sigma t}
\end{align*}
for all $t \ge 0$ and $x \in \partial \Omega$.
By $\nu\cdot\nabla \zeta  = 1$ on $\partial \Omega$ and \eqref{pq}, there holds
\begin{equation*}
\nu\cdot\nabla w_1^i(t,x)
=\nu\cdot \nabla V_i(x_1, z) - \delta q_i^1 e^{-\sigma t}
\le \delta q_* e^{-\sigma t} - \delta q^* e^{-\sigma t} = 0
\end{equation*}
for all $t \ge 0$ and $x \in \partial \Omega$.

Next, we check that
$
\mathcal L_i(w_{1})(t,x)
:= (w_{1}^i)_t(t,x) - D_i \Delta w_{1}^i(t,x) - F_i(w_{1}(t,x)) \le 0
$
for all $(t,x)\in[0,\infty)\times\Omega$.
By a direct calculation, we have
\begin{align*}
\mathcal L_i(w_1)(t,x)
={}& \sigma \delta \rho (V_z)(x_1,z) e^{-\sigma t}
+ \sigma \delta q_i(\xi_1)\zeta(x) e^{-\sigma t} \\
&- \delta q_i'(\xi_1)(-s + \sigma \rho \delta e^{-\sigma t})\zeta(x) e^{-\sigma t}
+ D_i \delta \Delta q_i(\xi_1)\zeta(x) e^{-\sigma t} \\
&+ D_i \delta q_i(\xi_1)\Delta \zeta(x) e^{-\sigma t}
+ 2 D_i \delta \nabla q_i(\xi_1)\cdot\nabla\zeta(x) e^{-\sigma t} \\
&+ F_i(V(x_1,z)) - F_i(w_1(t,x)).
\end{align*}

If $\xi_1\leq -C-m_0$, then by \eqref{psii} we have $z-m_*|x_1|\leq -C$.  
From \eqref{dpq}, it follows that $Q(\xi_1(t,x))=Q_1$.  
Using \eqref{pq} and \eqref{dll}, we obtain
\begin{align*}
(1-4\varepsilon_0)\cdot\mathbf1
\leq (1-\delta q_*)\cdot\mathbf1 - \delta Q_1\|\zeta\|_{L^\infty(\bar\Omega)}
\leq w_1(t,x)
\leq V(x_1,z)
\leq \mathbf1.
\end{align*}
By \eqref{muij}, there holds
\begin{align*}
F_i(V(x_1,z)) - F_i(w_1(t,x))
\leq \delta\sum_{j=1}^m \mu_{ij} q_j^1 \zeta(x) e^{-\sigma t}
\leq -\delta \varpi q_i^1 \zeta(x) e^{-\sigma t}.
\end{align*}
Since $V_z\ll\mathbf0$, by \eqref{D}, \eqref{zeta} and \eqref{si}, one has
\begin{align*}
\mathcal L_i(w_1)(t,x)
\leq \delta q_i^1 \zeta(x) e^{-\sigma t}
\left(\sigma + \bar D \left\|\frac{\Delta \zeta}{\zeta}\right\|_{L^\infty(\bar\Omega)} - \varpi\right)
\leq 0.
\end{align*}

If $\xi_1\geq C$, then by \eqref{psii} we have $z - m_*|x_1| \geq C + \psi(x_1) - m_*|x_1| \geq C$.  
Since $C>1$, it follows from \eqref{dpq} that $Q(\xi_1(t,x)) = Q_0$.  
By similar arguments, one has $\mathcal L_i(w_1)(t,x) \leq 0$.

If $-C-m_0\leq \xi_1 \leq C$, then $-C\leq z-m_*|x_1|\leq C$ by \eqref{psii}. 
From \eqref{pq} and \eqref{LA}, there holds
\begin{align*}
F_i (V(x_1,z) ) - F_i (w_1(t,x) )
\leq\delta\Lambda q_i(\xi_1) \zeta(x) e^{-\sigma t}
\leq\delta\Lambda q^*\|\zeta\|_{L^\infty(\bar\Omega)} e^{-\sigma t}.
\end{align*}
Similarly as in \eqref{naq} and \eqref{deq}, one has
$|\nabla q_i(\xi_1)|\le \frac{sM}{c}$
and
$|\Delta q_i(\xi_1)| \le \left(\frac{s^2}{c^2} + \|\psi''\|_{L^\infty(\mathbb R)}\right) M$.
By \eqref{pq}, \eqref{D}, \eqref{kapppp} and \eqref{rho}, there holds
\begin{align*}
\mathcal L_i(w_1)(t,x)
\leq{}& -\sigma \delta \rho \kappa e^{-\sigma t}
+ \sigma \delta q^*\|\zeta\|_{L^\infty(\bar\Omega)} e^{-\sigma t}
+s \delta M\|\zeta\|_{L^\infty(\bar\Omega)} e^{-\sigma t} \\
&+ \bar D \delta \left(\frac{s^2}{c^2} + \|\psi''\|_{L^\infty(\mathbb R)}\right) M\|\zeta\|_{L^\infty(\bar\Omega)} e^{-\sigma t}
+\delta\Lambda q^*\|\zeta\|_{L^\infty(\bar\Omega)} e^{-\sigma t} \\
&+ \bar D \delta q^*\|\Delta \zeta\|_{L^\infty(\bar\Omega)} e^{-\sigma t}
+ \frac{2}{c}\delta s\bar DM \|\nabla\zeta\|_{L^\infty(\bar\Omega)} e^{-\sigma t} \\
\leq{}&0.
\end{align*}
The proof is complete.
\end{proof}

An analogous argument to the proof of Lemma \ref{sub1} yields the following conclusion.

\begin{lem}\label{sup1}
Fix any $H\in\mathbb R$. For any small $\delta>0$, there exist $\sigma>0$, $\rho>0$ and $T>0$ such that
\[
\bar{ w}_1(t,x) = \min\{ w_1^+(t,x), \mathbf 1 \}
\]
is a supersolution of \eqref{1.1} for $t\ge 0$ and $x\in\bar\Omega$, where  $w_1^+ = (w_{1,1}^+,\dots,w_{1,m}^+)$ is defined by
\[
w_1^+(t,x) = V(x_1, z^+) + \delta Q(\xi_1^+(t,x)) \zeta(x) e^{-\sigma t}
\]
with
$
\xi_1^+(t,x) = z^+ - \psi(x_1)$
and 
$z^+ = x_2 - s(t-T) - \rho\delta(1-e^{-\sigma t}) + H$.
\end{lem}

Let $\gamma_1:=sm_*$ and
\begin{equation*}
  \varphi(\xi):=\frac{1}{m_*\gamma_1}\ln\left(1+\exp (\gamma_1\xi)\right).
\end{equation*}

\begin{lem}[\cite{NT1}]\label{tvar}
There exist some constants $K_i>0$ $(i=1,2,3)$ such that
\begin{align*}
&\max\left\{\left|\varphi(\xi)-\frac{\xi}{m_*}\right|, \left|\varphi'(\xi)-\frac{1}{m_*}\right|\right\} 
\le K_1 {\rm sech}(\gamma_1 \xi), &&  \xi \ge 0,\\
&\max\left\{\left|\varphi(\xi)\right|, \left|\varphi'(\xi)\right|\right\} 
\le K_1 {\rm sech}(\gamma_1 \xi), && \xi \le 0,\\
&\max\left\{\left|\varphi''(\xi)\right|, \left|\varphi'''(\xi)\right|\right\} 
\le K_1 {\rm sech}(\gamma_1 \xi), && \xi \in \mathbb R,\\
&c - \frac{s \varphi(\xi)}{\sqrt{1 + (\varphi'(\xi))^2}} 
\ge K_2 \min\{1, \exp(\gamma_1 \xi)\}, &&  \xi \in \mathbb R,\\
&0 \le \frac{s}{\sqrt{1 + (\varphi'(\xi))^2}} - c m_* 
\ge K_3 \min\{1, \exp(\gamma_1 \xi)\}, &&  \xi \in \mathbb R.
\end{align*}
\end{lem}

It is worth noting that the planar traveling wave front in \cite{wangzhicheng2012} is increasing, so a suitable change of variables is needed to apply their results in the present work.
Based on \cite[Lemma 3.2]{wangzhicheng2012} and the comparison principle, the following result follows immediately.
\begin{lem}\label{sub2}
There exist positive constants $\varepsilon_0^-$ and $\alpha^-(\varepsilon)$ such that for 
$0<\varepsilon<\varepsilon_0^-$ and $0<\alpha<\alpha^-(\varepsilon)$, the function
$$
\underline v_2(y,z;\varepsilon,\alpha)=\max\{v_2(y,z;\varepsilon,\alpha),\mathbf0\}
$$
is a subsolution of \eqref{Veq} for $(y,z)\in\mathbb R^2$, where
\begin{equation*}
v_2(y,z;\varepsilon,\alpha)
  := \Phi\left(\frac{\varphi(\alpha y)-\alpha z}{\alpha\sqrt{1+\varphi'(\alpha y)^2}}\right)
  - \varepsilon Q\left(\frac{\varphi(\alpha y)-\alpha z}{\alpha\sqrt{1+\varphi'(\alpha y)^2}}\right)
  {\rm sech}(\gamma_1 \alpha y).
\end{equation*}
Moreover, for any fixed constant  $C>0$, there exists  $\kappa_0(C)>0$ such that
\begin{equation*}
 (v_2)_y(y,z;\varepsilon,\alpha)
  \ge \kappa_0\cdot\mathbf1
  \quad \text{ if } 
  -C \le \frac{\varphi(\alpha y)-\alpha z}{\alpha\sqrt{1+\varphi'(\alpha y)^2}} \le C.
\end{equation*}
\end{lem}
 
Define 
\begin{equation*}
v_3(y,z;\varepsilon,\alpha) = v_2(-y,z;\varepsilon,\alpha).
\end{equation*}
Then $v_3(y,z;\varepsilon,\alpha)$ is also a subsolution of \eqref{Veq}. 
For simplicity, we shall denote $v_2(y,z;\varepsilon,\alpha)$ and $v_3(y,z;\varepsilon,\alpha)$  by $v_2(y,z)$ and $v_3(y,z)$ respectively in the following. 
Note that \eqref{muij} remains valid if we replace $\varepsilon_0$ by any $0 < \hat\varepsilon_0 < \varepsilon_0$. 
Therefore, by decreasing $\varepsilon_0$ if necessary and applying the comparison principle, 
the following result follows directly from \cite[Lemma 4.4]{wangzhicheng2012}.

\begin{lem}\label{sub3}
 For any small $\delta >0$ and $T \in \mathbb R$, 
there exist a  large constant $\rho>0$ and a  small  constant $\sigma>0$ such that
$$
\underline w_2(t,x)=\max\{w_2(t,x),\mathbf0\}\quad
\text{ and }\quad
\underline w_3(t,x)=\max\{w_3(t,x),\mathbf0\}
$$  
are subsolutions of \eqref{1.1} for $t \ge 0$ and $x\in\R^2$ (in the case that $\Omega=\R^2$) respectively,
where  $w_j(t,x)$ $(j=2,3)$ are defined by
\begin{align*}
w_2(t,x) &= v_2\left(x_1-\rho\delta(1-e^{-\beta t}), x_2-s(t+T)\right) - \delta Q(\xi_2) e^{-\sigma t},\\
w_3(t,x) &= v_3\left(x_1+\rho\delta(1-e^{-\beta t}), x_2-s(t+T)\right) - \delta Q(\xi_3) e^{-\sigma t}
\end{align*}
with
\begin{equation*}
\xi_2 = \frac{x_1-\rho\delta(1-e^{-\beta t}) + \varphi(\alpha \xi)/\alpha}{\sqrt{1 + (\varphi'(\alpha \xi))^2}}, \quad
\xi_3 = \frac{-x_1-\rho\delta(1-e^{-\beta t}) + \varphi(\alpha \xi)/\alpha}{\sqrt{1 + (\varphi'(\alpha \xi))^2}}.
\end{equation*}
\end{lem}

On the basis of  $w_1$, $w_2$ and $ w_3$, we construct a subsolution of \eqref{1.1} which plays a key role in the proof of Theorem \ref{t2}.

\begin{lem}\label{sub4}
For any small $\delta>0$ and any $H>0$, there exist constants $T>0$, $\rho>0$ and $\sigma>0$ such that
\begin{equation*}
\widetilde{ w}^-(t,x) := \max\left\{ w^-(t,x), \mathbf 0\right\}
\end{equation*}
is a subsolution of \eqref{1.1} for $t \ge 0$ and $x \in \bar\Omega$, where
\begin{equation*}
w^-(t,x) = \max\left\{ w_1(t,x),  w_2(t,x),  w_3(t,x)\right\}.
\end{equation*}
\end{lem}

\begin{proof}
Since the conclusion holds trivially when $\widetilde{w}^-(t,x) = \mathbf 0$, we only need to consider the case $\widetilde{w}^-(t,x) = w^-(t,x) $.  
By Lemmas \ref{sub1} and \ref{sub3}, it follows that 
\[
(w_i)_t - D \Delta w_i - F(w_i) \le \mathbf 0, \quad i = 1, 2, 3.
\]
Thus, it remains to verify the boundary condition for $w^-$.  
We claim that  
\begin{equation*}
w^-(t,x) = w_1(t,x) \quad \text{for all } x \in \partial \Omega \text{ and } t \ge 0.
\end{equation*}
Once this is shown, we immediately obtain $\nabla w^-(t,x) \cdot \nu \le 0$ for $t \ge 0$ and $x \in \partial \Omega$ from the proof of Lemma \ref{sub1}.
By Lemma \ref{tvar}, $\varphi(\xi) > 0$ and $|\varphi'(\xi)| < \infty$ for $\xi \in \mathbb R$.  
Then, by \eqref{ck} and the definition of $w_2$, there exists $0 < \pi < 1$ such that
\begin{equation*}
w_2(t,x) \le (1 - \pi) \cdot \mathbf 1
\quad \text{for any } T \in \mathbb R,\ t \ge 0,\ x \in \partial \Omega.
\end{equation*}
Similarly, $w_3(t,x) \le (1 - \pi) \cdot \mathbf 1$ for any $T \in \mathbb R$, $t \ge 0$ and $x \in \partial \Omega$.  
On the other hand, since 
\[
x_2 - s(t + T) + \rho \delta (1 - e^{-\sigma t}) + H \to -\infty
\quad \text{for } t \ge 0 \text{ and } x \in \partial \Omega \text{ as } T \to \infty,
\]
it follows from the properties of $V$ that
\begin{equation*}
w_1(t,x) \ge \left(1 - \delta \|\zeta\|_{L^{\infty}(\overline{\Omega})}\right) \cdot \mathbf 1
\quad \text{for } t \ge 0 \text{ and } x \in \partial \Omega \text{ as } T \to \infty.
\end{equation*}
Therefore, for sufficiently small $\delta > 0$, there exists a large $T > 0$ such that
\begin{equation*}
w_1(t,x) \ge (1 - \pi) \cdot \mathbf 1 \ge \max\{w_2(t,x),\ w_3(t,x),\mathbf0\}
\quad \text{for } t \ge 0 \text{ and } x \in \partial \Omega.
\end{equation*}  
The proof is complete.
\end{proof}

The following lemma follows from \cite[Lemma 3.1]{wangzhicheng2012} and the comparison principle.
\begin{lem} \label{super1}
There exist constants $\gamma_2 > 0$, $0 < \varepsilon_0^+ < 1$ and  $\alpha^+(\varepsilon)>0$ such that for 
$0 < \varepsilon < \varepsilon_0^+$ and $0 < \alpha < \alpha^+(\varepsilon)$, the function 
$$
\bar v(y,z;\varepsilon,\alpha)=\min\{v_4(y,z;\varepsilon,\alpha),\mathbf1\}
$$
is a supersolution of \eqref{Veq} for all $(y,z) \in \mathbb R^2$, where
\begin{equation*}
v_4(y,z;\varepsilon,\alpha)
= \Phi\left(\frac{\alpha z - \psi(\alpha y)}{\alpha\sqrt{1 + \psi'(\alpha y)^2}}\right)
+ \varepsilon P\left(\frac{\alpha z - \psi(\alpha y)}{\alpha\sqrt{1 + \psi'(\alpha y)^2}}\right)
{\rm sech}(\gamma_2 \alpha y).
\end{equation*}
 Moreover, $(v_4)_z(y,z;\varepsilon,\alpha) \ll \mathbf0$ for all $(y,z) \in \mathbb R^2$.  
Hence, for any constant $C > 0$, there exists $\kappa_1 = \kappa_1(C) > 0$ such that 
\[
(v_4)_z(y,z;\varepsilon,\alpha) \le- \kappa_1 \cdot \mathbf1
\quad \text{if} \quad
-C \le \frac{\alpha z - \psi(\alpha y)}{\alpha\sqrt{1 + \psi'(\alpha y)^2}} \le C.
\]
\end{lem}

We now proceed to construct a supersolution of \eqref{1.1} by employing $v^+$.
For simplicity, we   denote $v_4(y,z;\varepsilon,\alpha)$   by $v_4(y,z)$.

\begin{lem}\label{super2}
For any small $\delta > 0$, there exist constants $\rho > 0$, $\sigma > 0$ and $T > 0$ such that
\begin{equation*}
\widetilde{w}^+(t,x) := \min\left\{ w^+(t,x), \mathbf1 \right\}
\end{equation*}
is a supersolution of \eqref{1.1} for all $t \ge 0$ and $x \in \overline{\Omega}$, where 
\[
w^+(t,x)
= v_4(x_1, z)
+ \delta P(\xi(t,x)) e^{-\sigma t}
\]
with
\[
z = x_2 - s(t + T) - \rho \delta (1 - e^{-\sigma t}),
\qquad
\xi(t,x) = \frac{\alpha z - \psi(\alpha x_1)}{\alpha \sqrt{1 + \psi'(\alpha x_1)^2}}.
\]
\end{lem}
\begin{proof}
Since the conclusion is trivially true when $\widetilde{ w}^+ = \mathbf1$, we only need to consider the case where $\widetilde{ w}^+(t,x) =  w^+(t,x) $.  
In this case, it follows from \cite[Lemma 4.3]{wangzhicheng2012} that
$w^+_t- D \Delta w^+
- F(w^+)
\ge \mathbf0$.
In order to verify the boundary condition, we shall show that $ \widetilde w^+(t,x) = \mathbf1$ for all $t \ge 0$ and $x\in \partial \Omega$, namely,
\begin{equation}\label{pw}
w^+(t,x)
=  v_4(x_1,z)
+ \delta  P(\xi(t,x)) e^{-\sigma t}
\ge \mathbf1
\quad \text{ for all } t \ge 0 \text{ and }x\in \partial \Omega.
\end{equation}
By \eqref{ck}, one has $\varepsilon\ {\rm sech}(\gamma_2 \alpha x_1) > 0$ for all $x\in \partial \Omega$.  
Since
\[
\frac{\alpha z - \psi(\alpha x_1)}{\alpha \sqrt{1 + \psi'(\alpha x_1)^2}}
\to -\infty
\quad \text{ as } z \to -\infty
\]
and $\Phi(-\infty) = \mathbf1$, there exists a sufficiently large  $T > 0$ such that \eqref{pw} holds.  
This completes the proof.
\end{proof}
For any fixed $H \in \mathbb R$, let $v_1(y,z) :=  V(y, z + H)$ for $(y, z) \in \mathbb R^2$.  
Then $w_1$ can be expressed as
\begin{equation*}
w_1(t,x)
=  v_1\left(x_1, x_2 - s(t + T) + \rho \delta (1 - e^{-\sigma t})\right)
- \delta  Q(\xi_1(t,x)) \zeta(x) e^{-\sigma t}.
\end{equation*}
Define
\begin{equation*}
v^-(t,x)
= \max\left\{
v_1(x_1, x_2 - st),~
v_2(x_1, x_2 - st),~
v_3(x_1, x_2 - st)
\right\},
\end{equation*}
and
\begin{equation*}
v^+(t,x) = v_4(x_1, x_2 - st).
\end{equation*}

Applying arguments similar to those in \cite[Lemmas 3.1 and 4.10]{wangzhicheng2012}, we readily obtain the following result.

\begin{lem}\label{large1}
It holds that
\begin{equation*}
\lim\limits_{R \to \infty}
\sup\limits_{x_1^2 + (x_2 - st)^2 > R^2}
\left|
v^{\pm}(t,x)
- \Phi\left(\frac{c}{s}(x_2 - st - m_* |x_1|)\right)
\right|
\le \varepsilon q_*
\quad \text{for any } t \ge 0,
\end{equation*}
where $\varepsilon$ is a sufficiently small constant such that Lemmas \ref{sub2} and \ref{super1} hold.  
In particular, it follows from \eqref{1.b} that
\begin{equation*}
\lim\limits_{R \to \infty}
\sup\limits_{x_1^2 + (x_2 - st)^2 > R^2}
\left|
v^{\pm}(t,x)
- V(x_1, x_2 - st)
\right|
\le \varepsilon q_*
\quad \text{for any } t \ge 0.
\end{equation*}
\end{lem}

\subsection{Large-time behavior of entire solution}

Let $u(t,x)$ be the entire solution of \eqref{1.1} obtained in Theorem \ref{t1}.  
In addition, assume that $u(t,x)$ propagates completely in the sense of \eqref{1.6}.  
In the remainder of this section, we mainly focus on the proof of Theorem \ref{t2}.

\begin{lem}\label{U1}
For any small $\varepsilon>0$, there exist constants $\overline T \in \mathbb{R}$, $A_1>0$ and $A_2>0$ such that 
\begin{equation*}
u(t,x) \ge \left(1 - \varepsilon q_*\right) \cdot \mathbf1 
\quad \text{for all }t \ge \overline T\text{ and } x \in \overline{\Omega} \text{ with } x_2 - st - m_* |x_1| \le -A_1,
\end{equation*}
and
\begin{equation*}
u(t,x) \le \varepsilon q_* \cdot \mathbf1 
\quad \text{for all }t \ge \overline T\text{ and } x \in \overline{\Omega}  \text{ with } x_2 - st - m_* |x_1| \ge A_2.
\end{equation*}
\end{lem}

\begin{proof}
The proof is divided into two steps.

{\it Step 1:  proof of the first statement.}
Let $\varepsilon$ be any small constant such that Lemma \ref{large1} holds. 
Pick any small $\delta$ satisfying 
\begin{equation}\label{del}
0<\delta<\frac{\varepsilon q_*}{2q^*\|\zeta\|_{L^{\infty}(\overline\Omega)}}
\left(<\frac{\varepsilon}{2}\right).
\end{equation}
Take $H=0$ in Lemmas \ref{sub1} and \ref{sup1}.
It follows that there exist positive constants $\sigma$, $\rho$ and $T$ such that $\underline{w}_1(t,x)$ and $\overline{w}_1(t,x)$ are a pair of sub- and supersolutions of \eqref{1.1} for all $t\geq0$ and $x\in\bar\Omega$. 
By the properties of the V-shaped traveling front,
there exists a constant $R>0$ such that
\begin{equation}\label{vdq}
\begin{cases}
V(x_1,x_2-st)\leq\delta q_*\cdot\mathbf1 & \text{if }x_2-st-m_*|x_1|\geq R-sT,\\[4pt]
V(x_1,x_2-st)\geq\left(1-\delta q_*\right)\cdot\mathbf1 & \text{if }x_2-st-m_*|x_1|\leq -R-sT.
\end{cases}
\end{equation}
By Theorem \ref{t1}, there exists $\tilde T<0$ sufficiently small such that
\begin{equation}\label{t11}
-\tfrac{1}{2}\delta q_*\cdot\mathbf1
\leq 
u(\tilde T,x)-V(x_1,x_2-s\tilde T)
\leq
\tfrac{1}{2}\delta q_*\cdot\mathbf1
\quad\text{for }x\in\overline\Omega.
\end{equation}
By \eqref{1.b} and $\Phi(-\infty)=\mathbf1$, there exists $\tilde A>0$ such that
\begin{equation}\label{uV}
u(\tilde T,x)\geq V(x_1,x_2-s\tilde T)-\tfrac{1}{2}\delta q_*\cdot\mathbf1
\geq (1-\delta q_*)\cdot\mathbf1
\end{equation}
for $x\in\overline\Omega$ such that $x_2-s\tilde T-m_*|x_1|\leq-\tilde A$. 

It is evident that for any point $\tilde{y}=(\tilde y_1,\tilde y_2)\in\overline\Omega$ with $\tilde y_2\leq m_*|\tilde y_2|+R$, there holds 
\[
d_\Omega\left(\tilde{y},\left\{x\in\overline\Omega:x_2-m_*|x_1|-s\tilde T\leq-\tilde A\right\}\right)<\infty.
\]
Since the obstacle $K$ is compact, there exist $x_0\in\mathbb R^2$ and $L>0$ such that $K\subset B(x_0,L)$. 
It follows from Lemma \ref{key} that there are positive constants $R_1$, $R_2$, $R_3$ and $\hat T>0$ such that $R_3>R_2>R_1>0$, and if $B(x_0,R_3)\subset\Omega$, then
\[
v_{x_0,R_1}^\delta(\hat T,\cdot)\geq(1-\delta q_*)\cdot\mathbf1
\quad\text{in }\overline{B(x_0,R_2)},
\]
where $v_{x_0,R}^\delta$ is the solution of \eqref{vr}.
Then, for any point 
$
y\in\overline{\Omega\backslash B(x_0,L+R_3-R_2)}\cap
\left\{x\in\mathbb R^2:x_2-m_*|x_1|\leq R\right\}$,
there exist $k$ points $x^1,\dots,x^k\in\mathbb R^2$ such that
\begin{equation*}
\begin{cases}
B(x^1,R_1)\subset\{x\in\mathbb R^2:x_2-m_*|x_1|-s\tilde T\leq -\tilde A\},\\ 
B(x^i,R_3)\subset\Omega, & 1\leq i\leq k,\\ 
B(x^{i+1},R_1)\subset B(x^i,R_2), & 1\leq i\leq k-1,\\ 
y\in B(x^k,R_2).
\end{cases}
\end{equation*}
It follows from Lemma \ref{key}, \eqref{uV} and the comparison principle that
\begin{equation*}
 u(\tilde T+\hat T,x)\geq  v_{  x^1,R_1}(\hat T,x)\geq (1-\delta q_*)\cdot\mathbf1\quad\text{ for }x\in\overline{
B( x^1,R_2)}.
\end{equation*}
Since $B( x^{2},R_1)\subset B( x^{1},R_2)$, one gets that $ u(\tilde T+\hat T,x)\geq (1-\delta q_*)\cdot\mathbf1
$ for $x\in\overline{B(x^2,R_1)}$. Since $B( x^2,R_3)\subset\Omega$, one apply Lemma \ref{key} again and get that $
u(\tilde T+2\hat T,x)\geq (1-\delta q_*)\cdot\mathbf1$ for $x\in\overline{B( x^2,R_2)}$. By induction, one has that $
 u(\tilde T+k\hat T,x)\geq (1-\delta q_*)\cdot\mathbf1$ for $x\in\overline{B( x^k,R_2)}$. Since $y\in B( x^k,R_2)$, there holds
\begin{align}\label{ukT}
u(\tilde T+ k\hat T,x )\geq&(1-\delta q_*)\cdot\mathbf1\\
&\text{ in }\ \overline{\Omega\backslash B(x_0,L+R_3-R_2)}\cap\left\{x\in\mathbb R^2:x_2- m_*|x_1|\leq R\right\}.\notag
\end{align}
It follows from \eqref{1.6} that there is $T{'}\in\mathbb R$ large enough such that
\begin{equation}\label{ukdk}
u(\tilde T+ T{'},x)\geq(1-\delta q_*)\cdot\mathbf1\quad\text{ for all }x\in\overline{ B(x_0,L+R_3-R_2)\backslash  K}.
\end{equation}

Define 
\[
\overline T = \max\left\{\tilde T + k\hat T,\, \tilde T + T'\right\}.
\]
Since $u_t \gg \mathbf0 $ from Theorem \ref{t1}, it follows from \eqref{ukT} and \eqref{ukdk} that
\begin{equation*}
u(\overline T, x) \ge (1 - \delta q_*) \cdot \mathbf1
\quad \text{for all } 
x \in \overline{\Omega} \text{ such that } x_2 - m_* |x_1| \le R.
\end{equation*}
Since $\mathbf0 \le V \le \mathbf1$ and $\zeta \ge 1$, we have
\begin{equation*}
u(\overline T, x) 
\ge (1 - \delta q_*) \cdot \mathbf1
\ge V(x_1, x_2 - sT) - \delta Q(\xi_1(0, x)) \zeta(x)
= w_1(0, x)
\end{equation*}
for all $x \in \overline{\Omega}$ satisfying $x_2 - m_* |x_1| \le R$.
For any $x \in \overline{\Omega}$ such that $x_2 - m_* |x_1| \ge R$, we have $x_2 - sT - m_* |x_1| \ge R - sT$.  
Hence, by \eqref{pq} and   \eqref{vdq}, there holds
\begin{equation*}
w_1(0, x)
= V(x_1, x_2 - sT) - \delta Q(\xi_1(0, x)) \zeta(x)
\le (\delta q_*- \delta q_*)\cdot\mathbf1
=\mathbf0
\le u(\overline T, x)
\end{equation*}
for all $x \in \overline{\Omega}$ with $x_2 - m_* |x_1| \ge R$.

As a conclusion,
\[
u(\overline T, x) \ge \max\{ w_1(0, x), \mathbf0 \}=\underline w_1(0,x)
\quad \text{for all } x \in \overline{\Omega}.
\]
By the comparison principle, it follows that
\begin{equation*}
u(\overline T + t, x) 
\ge \underline{w}_1(t, x)
= \max\{ w_1(t, x), \mathbf0 \}
\ge w_1(t, x)
\quad \text{for all } t \ge 0 \text{ and } x \in \overline{\Omega}.
\end{equation*}
From \eqref{del} and the properties of the V-shaped traveling front, we obtain that there exists a constant $A_1 > 0$ such that
\begin{align*}
u(t, x)
&\ge w_1(t - \overline T, x) \\
&= V\left(x_1,\, x_2 - s(t - \overline T + T) + \rho \delta \big(1 - e^{-\sigma (t - \overline T)}\big)\right) 
- \delta Q(\xi_1(t - \overline T, x)) \zeta(x) e^{-\sigma (t - \overline T)} \\
&\ge \left(1 - \tfrac{1}{2} \varepsilon q_*\right) \cdot \mathbf1 
- \delta q^*\|\zeta\|_{L^\infty(\bar\Omega)} \cdot \mathbf1 
\ge \left(1 - \varepsilon q_*\right) \cdot \mathbf1
\end{align*}
for all $t \ge \overline T$ and $x \in \overline{\Omega}$ such that $x_2 - st - m_* |x_1| \le -A_1$.  
This completes the proof of the first statement of this lemma.

{\it Step 2: proof of the second statement.}
By \eqref{1.b}, \eqref{t11} and $\Phi(\infty) = \mathbf0$, there exists $\bar A > 0$ such that
\begin{equation*}
u(\tilde T, x) \le V(x_1, x_2 - s\tilde T) + \tfrac{1}{2}\delta q_* \cdot \mathbf1 
\le \delta q_* \cdot \mathbf1
\end{equation*}
for all $x \in \overline{\Omega}$ such that $x_2 - m_* |x_1| - s\tilde T \ge \bar A$.  
From \eqref{pq}, there holds
\begin{equation*}
w_1^+(0, x) = V(x_1, x_2 - sT) + \delta Q(\xi_1(0, x)) \zeta(x)
\ge \delta q_* \cdot \mathbf1 \ge u(\tilde T, x)
\end{equation*}
for all $x \in \overline{\Omega}$ such that $x_2 - m_* |x_1| - s\tilde T \ge \bar A$.

For any $x \in \overline{\Omega}$ such that $x_2 - m_* |x_1| - s\tilde T \le \bar A$, one has
$x_2 - sT - m_* |x_1| \le s(\tilde T - T) + \bar A$.
Since $T > 0$ and $\tilde T < 0$ is sufficiently small, by decreasing $\tilde T$ if necessary, we can ensure that 
$x_2 - sT - m_* |x_1| \le -R - sT$.  
It then follows from \eqref{vdq} that
\begin{equation*}
V(x_1, x_2 - sT) \ge (1 - \delta q_*) \cdot \mathbf1 
\quad \text{for all } x \in \overline{\Omega} 
\text{ such that } x_2 - m_* |x_1| - s\tilde T \le \bar A.
\end{equation*}
Therefore, one has
\begin{equation*}
w_1^+(0, x) = V(x_1, x_2 - sT) + \delta Q(\xi_1^+(0, x)) \zeta(x)
\ge (1 - \delta q_*) \cdot \mathbf1 +\delta q_* \cdot \mathbf1 
= \mathbf1 \ge u(\tilde T, x)
\end{equation*}
for all $x \in \overline{\Omega}$ such that $x_2 - m_* |x_1| - s\tilde T \le \bar A$.  

As a conclusion, we have 
\[
\bar w_1(0,x)=\min\{ w_1^+(0, x), \mathbf1 \} \ge u(\tilde T, x)
\quad \text{for all } x \in \overline{\Omega}.
\]
By the comparison principle, it follows that
\begin{equation*}
u(t+\tilde T , x) 
\le\bar w_1(t,x)=\min\{ w_1^+(t, x), \mathbf1 \} 
\le w_1^+(t, x)
\quad \text{for all } t \ge 0 \text{ and } x \in \overline{\Omega}.
\end{equation*}
It follows from \eqref{1.b}, \eqref{del} and $\Phi(\infty) = \mathbf0$ that there exists $A_2 > 0$ such that
\begin{align*}
u(t, x) 
&\le w_1^+(t - \tilde T, x) \\
&= V\left(x_1, x_2 - s(t - \tilde T + T) - \rho \delta (1 - e^{-\sigma t})\right)
  + \delta Q(\xi_1(t, x)) \zeta(x) e^{-\sigma t} \\
&\le \tfrac{1}{2} \varepsilon q_* \cdot \mathbf1 
  + \delta q^*\|\zeta\|_{L^\infty(\bar\Omega)} \cdot \mathbf1 
  \le \varepsilon q_* \cdot \mathbf1
\end{align*}
for all $t \ge \tilde T$ and $x \in \overline{\Omega}$ such that 
$x_2 - m_* |x_1| - st \ge A_2$.
Since $\bar T>\tilde T$, the second statement is true.
This completes the proof.
\end{proof}

Next, we show the relation between $ u$ and the V-shaped traveling front.
\begin{lem}{\label{u-V}}
Let $\overline T\in\mathbb R$ such that Lemma \ref{U1} holds. Then for any $\varepsilon>0$, there exists $\tilde R>0$ such that
\begin{equation*}
-\varepsilon q_*\cdot\mathbf1 \leq u(t,x )- V(x_1,x_2-st)\leq\varepsilon q_*\cdot\mathbf1
\end{equation*}
 for all $x\in\overline{\Omega \backslash B(\eta(t),\tilde R)}$, where $\eta (t)=(0,st)\in\mathbb R^2$.
\end{lem}

\begin{proof}
Let $r$ be any fixed constant. Choose a sequence 
$\{x_n\}_{n \in \mathbb N} = \{(x_{1n}, x_{2n})\}_{n \in \mathbb N} \subset \mathbb R^2$ 
such that 
\[
x_{1n} > 0, \quad 
x_{2n} - m_* x_{1n} = r 
\quad \text{and} \quad 
|x_n| \to \infty\ \text{ as } \ n \to \infty.
\]
Denote 
\[
u_n(t, x) =u(t,x+x_n)= u(t, x_1 + x_{1n},\, x_2 + x_{2n})
\]
for all $t \in \mathbb R$ and $x \in \Omega \setminus \{x_n\}$.  
Since $\mathbf0 \le u \le \mathbf1$ and the obstacle 
$K = \mathbb R^2 \setminus \Omega$ is bounded, it follows from standard parabolic estimates that, up to extraction of a subsequence, 
\[
u_n(t, x) \to U(t, x)
\quad \text{locally uniformly in } (t, x) \in \mathbb R \times \mathbb R^2
\text{ as } n \to \infty,
\]
where $U(t, x) \in [\mathbf0, \mathbf1]$ is a classical solution of
\begin{equation}\label{Ueq}
U_t - D \Delta U - F(U) = \mathbf0,
\quad t \in \mathbb R,~ x \in \mathbb R^2.
\end{equation}

By the choice of $\{x_n \}$, we have $x_{1n} \to \infty$.  
Hence, for $x_1 \ge -\frac{x_{1n}}{2}$,
\[
(x_1 + x_{1n})^2 + (x_2 + x_{2n} - st)^2 \to \infty
\quad \text{as } n \to \infty.
\]
It then follows from \eqref{1.b} that
\begin{equation*}
V(x_1 + x_{1n},\, x_2 + x_{2n} - st)
- \Phi\left(\tfrac{c}{s}(x_2 - st - m_* x_1 + r)\right)
\to \mathbf0
\end{equation*}
uniformly in $x \in \mathbb R^2$ as  $n \to \infty$.  
Moreover, since 
$u(t, x) - V(x_1, x_2 - st) \to \mathbf0$ 
uniformly in $x \in \mathbb R^2$ as $t \to -\infty$ from Theorem \ref{t1}, 
we obtain
\begin{equation}\label{UP}
U(t, x) \to 
\Phi\left(\tfrac{c}{s}(x_2 - st - m_* x_1 + r)\right)
\quad \text{uniformly in } x \in \mathbb R^2
\text{ as } t \to -\infty.
\end{equation}
We claim that
\begin{equation}\label{Claim}
U(t, x) =
\Phi\left(\tfrac{c}{s}(x_2 - st - m_* x_1 + r)\right)
\quad \text{for all } t \in \mathbb R \text{ and } x \in \mathbb R^2.
\end{equation}
For clarity and coherence, we place the proof of claim \eqref{Claim} after that of Lemma~\ref{u-V}.  

By \eqref{Claim}, there holds
\begin{equation}\label{unp}
u(t, x + x_n)
\to 
\Phi\left(\tfrac{c}{s}(x_2 - st - m_* x_1 + r)\right)
\end{equation}
locally uniformly in $(t, x) \in \mathbb R \times \mathbb R^2$ as $n \to \infty$.  
Similarly, for a sequence $\{x_n\}_{n \in \mathbb N}
= \{(x_{1n}, x_{2n})\}_{n \in \mathbb N} \subset \mathbb R^2$
such that 
$x_{1n} < 0$, 
$x_{2n} + m_* x_{1n} = r$
and
$|x_n| \to \infty $ as $ n \to \infty$,
there holds
\begin{equation*}
u(t, x + x_n)
\to 
\Phi\left(\tfrac{c}{s}(x_2 - st + m_* x_1 + r)\right)
\end{equation*}
locally uniformly in $(t, x) \in \mathbb R \times \mathbb R^2$ as $n \to \infty$.

Fix any $t_* \ge \overline T$.  
Pick a sequence $\{y_n\}_{n \in \mathbb N} = \{(y_{1n}, y_{2n})\}_{n \in \mathbb N} \subset \mathbb R^2$ 
satisfying $y_{1n} \ge 0$, $y_{2n} - m_* y_{1n} = st_*$ and $|y_n| \to \infty$ as $n \to \infty$.  
Notice that there exists $R > 0$ such that
\begin{equation*}
\Big\{x \in \overline\Omega : x_1 \ge 0,\ 
-A_1 \le x_2 - st_* - m_* x_1 \le A_2 \Big\}
\subset \bigcup_{n=1}^{\infty} B(y_n, R)
\end{equation*}
for some positive constants $A_1$ and $A_2$.  
It then follows from \eqref{unp} that
$
-\varepsilon q_* \cdot \mathbf1 \le 
u(t_*, x + y_n)
- \Phi\left(\tfrac{c}{s}(x_2 - m_* x_1)\right)
\le \varepsilon q_* \cdot \mathbf1
$
for all large $n$ and $|x| < R$.  
It implies that there exists $\tilde R > 0$ such that
\begin{equation}\label{ups}
-\varepsilon q_* \cdot \mathbf1 \le 
u(t_*, x)
- \Phi\left(\tfrac{c}{s}(x_2 - st_* - m_* x_1)\right)
\le \varepsilon q_* \cdot \mathbf1
\end{equation}
for all  $x \in \overline\Omega$ such that $x_1 \ge 0$, 
$-A_1 \le x_2 - st_* - m_* x_1 \le A_2$ 
and $|x - (0, st_*)| \ge \tilde R$.  
Similarly, by increasing $\tilde R$ if necessary,
one obtains that \eqref{ups} holds for all 
$x \in \overline\Omega$ such that $x_1 < 0$, 
$-A_1 \le x_2 - st_* + m_* x_1 \le A_2$
and $|x - (0, st_*)| \ge \tilde R$.  
According to the choice of $\tilde R$, 
one may take $A_1$ and $ A_2$ sufficiently large.

Since $t_* \ge \overline T$ was arbitrary, we deduce that
\begin{equation}\label{uPH1}
-\varepsilon q_* \cdot \mathbf1 \le
u(t, x)
- \Phi\left(\tfrac{c}{s}(x_2 - st - m_* |x_1|)\right)
\le \varepsilon q_* \cdot \mathbf1
\end{equation}
for all $t \ge \overline T$ and $x \in \overline\Omega$ such that
$-A_1 \le x_2 - st - m_* |x_1| \le A_2$
and $|x - \eta(t)| \ge \tilde R$, where $\eta(t) = (0, st)$.  
Since $\Phi'(\xi) \ll \mathbf0$, 
one can choose $A_1 > 0$ and $A_2 > 0$ sufficiently large so that
\begin{equation*}
\Phi\left(\tfrac{c}{s}(x_2 - st - m_* |x_1|)\right)
\ge (1 - \varepsilon q_*) \cdot \mathbf1
\end{equation*}
for all $t \ge \overline T$ and $x \in \overline\Omega$ such that $x_2 - st - m_* |x_1| \le -A_1$, 
and
\begin{equation*}
\Phi\left(\tfrac{c}{s}(x_2 - st - m_* |x_1|)\right)
\le \varepsilon q_* \cdot \mathbf1
\end{equation*}
for all $t \ge \overline T$ and  $x \in \overline\Omega$ such that $x_2 - st - m_* |x_1| \ge A_2$.  
Combining this with Lemma~\ref{U1}, we obtain
\begin{equation}\label{UPH2}
-\varepsilon q_* \cdot \mathbf1 \le 
u(t, x)
- \Phi\!\left(\tfrac{c}{s}(x_2 - st - m_* |x_1|)\right)
\le \varepsilon q_* \cdot \mathbf1
\end{equation}
for all $t \ge \overline T$ and $x \in \overline\Omega$ such that 
$x_2 - st - m_* |x_1| \le -A_1$ or $x_2 - st - m_* |x_1| \ge A_2$.  
It then follows from \eqref{1.b}, \eqref{uPH1}, and \eqref{UPH2} that
\begin{equation*}
-\varepsilon q_* \cdot \mathbf1 \le 
u(t, x) - V(x_1, x_2 - st)
\le \varepsilon q_* \cdot \mathbf1
\end{equation*}
for all $t \ge \overline T$ and 
$x \in \overline{\Omega \setminus B(\eta(t), \tilde R)}$.
The proof of Lemma \ref{u-V} is thereby complete.

{\it Proof of claim \eqref{Claim}.} The main strategy is to construct a pair of sub- and supersolutions of \eqref{Ueq}. 
Take 
\begin{align}\label{so}
0<\sigma<2\varepsilon_0
\ \text{ and }\
0<\omega<\frac{\varpi q_*}{q^*}(<\varpi).
\end{align}
Define
\begin{align}\label{rrho}
\rho=(\Lambda+\omega)q^*+M(c+\overline D)+1.
\end{align}
Since $\Phi(-\infty)=\mathbf1$ and $\Phi(\infty)=\mathbf0$, there exists a constant $C>1$ such that
\begin{equation}\label{PCC}
\begin{cases}
 \Phi(\xi)\geq(1-\sigma)\cdot\mathbf1 &\text{for all }\xi\leq -C,\\[3pt]
 \Phi(\xi)\leq\sigma\cdot\mathbf1 &\text{for all }\xi\geq C.
 \end{cases}
\end{equation}
Furthermore, there exists $\kappa>0$ such that
\begin{equation}\label{kappa}
\max_{i=1,\cdots,m}\left(\sup_{-C\leq\xi\leq C}\Phi_i'(\xi)\right)\leq -\kappa.
\end{equation}
Let $\varepsilon\in\left(0,\frac{2\varepsilon_0}{q^*}\right)$ be any fixed constant.  
By \eqref{UP}, there exists $T'<0$ such that
\begin{equation}\label{UPUP}
-\varepsilon q_*\cdot\mathbf1
\leq
U(t,x)-\Phi\!\left(\frac{c}{s}(x_2-st-m_*x_1+r)\right)
\leq
\varepsilon q_*\cdot\mathbf1
\end{equation}
for all $t\leq T'$  and $x\in\mathbb R^2$.
Choose any $T''\leq T'$.  
Define $u^\mp=(u_1^\mp,\cdots,u_m^\mp)$ by
\begin{equation*}
u^-(t,x)=\max\Bigl\{\Phi(\xi^-(t,x))-\varepsilon Q(\xi^-(t,x))e^{-\omega(t-T'')},\mathbf0\Bigr\}
\end{equation*}
and
\begin{equation*}
u^+(t,x)=\min\Bigl\{\Phi(\xi^+(t,x))+\varepsilon Q(\xi^+(t,x))e^{-\omega(t-T'')},\mathbf1\Bigr\}
\end{equation*}
for all $t\geq T''$ and $x\in\mathbb R^2$, where
\begin{equation*}
\xi^{\mp}(t,x)
=\frac{c}{s}(x_2-st-m_*x_1+r)
\pm\frac{\varepsilon \rho}{\omega\kappa}\bigl(1-e^{-\omega(t-T'')}\bigr).
\end{equation*}
We only need to prove that $u^-$ is a subsolution of \eqref{Ueq}, the proof that $u^+$ is a supersolution being similar.

Fix any index $i=1,\cdots,m$.  
Let us first show that 
\[
\mathcal L_i(u^-)(t,x) := (u^-_i)_t(t,x) - D_i \Delta u^-_i(t,x) - F_i(u^-(t,x)) \leq 0
\] 
for all $t \geq T''$ and $x \in \mathbb R^2$ such that $u_i^-(t,x) > 0$.  
Note that $\frac{c^2(1+m_*^2)}{s^2} = 1$ by \eqref{mstar}.  
By direct calculations, one has
\begin{align*}
\mathcal L_i(u^-)(t,x)
&= \varepsilon \omega q_i(\xi^-(t,x)) e^{-\omega(t-T'')}
- \varepsilon D_i q_i''(\xi^-(t,x)) e^{-\omega(t-T'')} \\ 
&\quad - \varepsilon q_i'(\xi^-(t,x)) \left(-c + \frac{\varepsilon \rho}{\kappa} e^{-2\omega(t-T'')}\right)
 + \frac{\varepsilon \rho}{\kappa} \Phi_i'(\xi^-(t,x)) e^{-\omega(t-T'')}\\
&\quad 
+ F_i(\Phi(\xi^-(t,x))) - F_i(u^-(t,x)). 
\end{align*}

If $\xi^-(t,x) \leq -C$, then $Q(\xi^-(t,x)) = Q_1$ from \eqref{dpq}.  
From \eqref{pq}, \eqref{so} and \eqref{PCC}, one has
\[
(1-4\varepsilon_0)\cdot\mathbf1 \leq (1-\sigma)\cdot\mathbf1 - \varepsilon q^* \cdot\mathbf1 \leq u^-(t,x) \leq \Phi(\xi^-(t,x)) \leq \mathbf1.
\]  
By \eqref{muij}, there holds
\begin{align*}
F_i(\Phi(\xi^-(t,x))) - F_i(u^-(t,x))
\leq \varepsilon \sum_{i,j=1}^m \mu_{ij}^1 q_j^1 e^{-\omega(t-T'')}
\leq - \varepsilon \varpi q_i^1 e^{-\omega(t-T'')}.
\end{align*}  
Hence, it follows from $\Phi_i' < 0$ and \eqref{so} that
\[
\mathcal L_i(u^-)(t,x)
= \varepsilon q_i^1 e^{-\omega(t-T'')} (\omega - \varpi) \leq 0.
\]

If $\xi^-(t,x) \geq C$, then $Q(\xi^-(t,x)) = Q_0$ from \eqref{dpq}.  
A similar argument implies that $\mathcal L_i(u^-)(t,x) \leq 0$.

Next, consider the case $-C \leq \xi^-(t,x) \leq C$.  
By \eqref{pq} and \eqref{LA}, there holds
\[
F_i(\Phi(\xi^-(t,x))) - F_i(u^-(t,x))
\leq \varepsilon \Lambda q_i(\xi^-(t,x)) e^{-\omega(t-T'')}
\leq \varepsilon \Lambda q^* e^{-\omega(t-T'')}.
\]  
By \eqref{pq}, \eqref{D}, \eqref{kappa} and \eqref{rrho}, one has
\[
\mathcal L_i(u^-)(t,x)
\leq \varepsilon e^{-\omega(t-T'')} 
\left(-\rho + \Lambda q^* + \omega q^* + c M + \bar D M \right) \leq 0.
\]  
Thus we have proved that $u^-$ is a subsolution of \eqref{Ueq} in $[T'',\infty) \times \mathbb R^2$.

Since \eqref{UPUP} implies that
\[
u^-(T'',x) \leq U(T'',x) \leq u^+(T'',x)
\quad\text{for all } x \in \mathbb R^2,
\]
it follows from the comparison principle that
\[
u^-(t,x) \leq U(t,x) \leq u^+(t,x)
\quad\text{for all } t \geq T'' \text{ and } x \in \mathbb R^2.
\]

Letting $T'' \to -\infty$, one deduces that
\[
\Phi\!\left(\frac{c}{s}(x_2 - st - m_* x_1 + r) + \frac{\varepsilon \rho}{\kappa \omega}\right)
\leq U(t,x) 
\leq \Phi\!\left(\frac{c}{s}(x_2 - st - m_* x_1 + r) - \frac{\varepsilon \rho}{\kappa \omega}\right)
\]
for all $t \in \mathbb R$ and $x \in \mathbb R^2$.  
By the arbitrariness of $\varepsilon$, it follows that the claim \eqref{Claim} is true.  
This completes the proof.
\end{proof}

\begin{proof}[Proof of Theorem \ref{t2}] 
The proof is divided into two steps.

{\it Step 1: proof of \eqref{1.7}.}
Let $\varepsilon>0$ be a sufficiently small constant. Let $\delta>2\varepsilon$ be a small constant such that Lemmas \ref{sub4} and \ref{super2} hold.
 Take $T>0$ large enough such that Lemmas \ref{sub4}, \ref{super2}, \ref{U1} and \ref{u-V} hold for $\delta$ and $\varepsilon$. It follows from Lemmas \ref{large1} and \ref{u-V} that there is $R>0$ such that
\begin{align*}
 u(T,x)&\geq\max\left\{ V(x_1,x_2-sT)-\varepsilon q_*\cdot\mathbf1,\mathbf0\right\}
\geq\max\left\{ v^-(T,x)-2\varepsilon q_*\cdot\mathbf1,\mathbf0\right\}\\
&\geq\max\left\{ v^-(T,x)-\delta q_*\cdot\mathbf1,\mathbf0\right\}
\geq\max\left\{ w^-(0,x),\mathbf0\right\}=\widetilde{ w}^-(0,x)
\end{align*}
and
\begin{align*}
 u(T,x)&\leq\min\left\{ V(x_1,x_2-sT)+\varepsilon q_*\cdot\mathbf1,\mathbf1\right\}
\leq\min\left\{ v^+(T,x)+2\varepsilon q_*\cdot\mathbf1,\mathbf1\right\}\\
&\leq\min\left\{ v^+(T,x)+\delta q_*\cdot\mathbf1,\mathbf1\right\}
\leq\min\left\{ w^+(0,x),\mathbf1\right\}=\widetilde{ w}^+(0,x)
\end{align*}
for all $x\in\overline\Omega$ such that $|x-(0,sT)|\geq R$.

Without loss of generality, it can be assumed that $ v_1(x_1,x_2-sT)\leq \delta q_*\cdot\mathbf1$ for all $x\in\overline\Omega$ such that $|(x_1,x_2)-(0,sT)|<R$ by increasing $H$ if necessary. Furthermore, it follows from the proof of \cite[Lemma 3.2]{NT1} that one can make $\alpha>0$ sufficiently small such that
\begin{equation*}
 v_2(x_1,x_2-sT;\varepsilon,\alpha), v_3(x_1,x_2-sT;\varepsilon,\alpha)\leq\delta q_*\cdot\mathbf1
\end{equation*}
for all $x\in\overline\Omega$ such that $|x-(0,sT)|< R$. 
It yields that $ w^-(0,x )\leq\mathbf0$ for $x\in\overline\Omega$ such that $|x-(0,sT)|<R$.
  Hence, $\widetilde{ w}^-(0,x )\leq u(T,x )$ for $x\in \overline\Omega$. One can obtain from the comparison principle that
\begin{equation*}
\widetilde{ w}^-(t,x )\leq u(t+T,x )\quad\text{ for all }t\geq0\text{ and }
x\in\overline\Omega.
\end{equation*}

Similarly, it follows from the proof of \cite[Theorem 1.3]{NT} that one can make $\alpha>0$ sufficiently small such that
\begin{equation*}
 v^+(x_1,x_2-sT;\varepsilon,\alpha)\geq(1-\delta q_*)\cdot\mathbf1
\end{equation*}
for all $x\in\overline\Omega$ such that $|x-(0,sT)|< R$.
It implies that  $\widetilde{ w}^+(0,x )=\mathbf1\geq u(T,x)$ for $x\in\overline\Omega$ such that $|x-(0,sT)|< R$.
Hence $\widetilde{ w}^+(0,x)\geq u(T,x )$ for all $x\in \overline\Omega$. 
By the comparison principle, there holds
\begin{equation*}
\widetilde{ w}^+(t,x )\geq u(t+T,x)\quad \text{ for all }t\geq0\text{ and }
x\in\overline\Omega.
\end{equation*}
In conclusion, one has
\begin{equation*}
w^-(t,x) \leq \widetilde w^-(t,x) \leq u(t+T,x) \leq \widetilde w^+(t,x) \leq w^+(t,x)
\quad \text{ for all } t \geq 0 \text{ and } x \in \overline\Omega.
\end{equation*}
Let $t+T \to \infty$, there holds 
\begin{align}\label{uv123}
&\max\Bigl\{
v_1(x_1, x_2 - st + \rho\delta;\varepsilon,\alpha),\ 
v_2(x_1 - \rho\delta, x_2 - st;\varepsilon,\alpha),\\
&\qquad
v_3(x_1 + \rho\delta, x_2 - st;\varepsilon,\alpha)
\Bigr\}
\leq u(t,x)
\leq v_4(x_1, x_2 - st - \rho\delta;\varepsilon,\alpha).\nonumber
\end{align}

Choose any sequence $\{t_n\}_{n\in\mathbb N} \subset \mathbb R$ such that $t_n \to \infty$ as $n \to \infty$.  
Define
\[
u_n(t,x) = u(t+t_n, x_1, x_2 + s t_n)
\quad \text{for } t \in \mathbb R \text{ and } x \in \overline\Omega.
\]
From standard parabolic estimates, up to extraction of a subsequence,
\[
u_n(t,x) \to U(t,x)
\quad \text{locally uniformly in } (t,x) \in \mathbb R \times \overline\Omega
\text{ as } n \to \infty,
\]
where $U(t,x)$ is the classical solution of 
\[
U_t - D\Delta U - F(U) = 0
\quad \text{for } t \in \mathbb R \text{ and } x \in \mathbb R^2.
\]

By \eqref{uv123} and the arbitrariness of $\varepsilon$ and $\delta$, one has
\[
v^-(t,x) \leq U(t,x) \leq v^+(t,x)
\quad \text{for all } t \in \mathbb R \text{ and } x \in \mathbb R^2.
\]
By Lemma~\ref{large1}, one has
\[
\lim_{R \to \infty} 
\sup_{x_1^2 + (x_2 - st)^2 > R^2}
\left| U(t,x) - \Phi\!\left(\frac{c}{s}\bigl(x_2 - st - m_*|x_1|\bigr)\right) \right|
= 0.
\]
The stability of the V-shaped front (see \eqref{stable}) implies that
\[
\lim_{t \to \infty} 
\| U(t,x) - V(x_1, x_2 - st) \|_{C^2(\mathbb R)} = 0.
\]
Therefore, for any $t > 2t_n$, one has
\[
\lim_{n \to \infty}
\| u_n(t - t_n, x) - V(x_1, x_2 - st) \|_{C^2(\mathbb R^2)} = 0.
\]
Since $u_n(t-t_n,x)=u(t,x)$ for each $n\in\mathbb N$, the formula \eqref{1.7} follows.

{\it Step 2: proof of \eqref{UP}.}
Let $\tilde\lambda_1$ and $d_1$ be the positive constants defined as in \eqref{V22}. 
Let $\vp>0$ be an arbitrary constant satisfying
\begin{align}\label{vps}
0<\vp<\min\left\{\frac{2\vp_0}{q^*},1\right\}<\left(\frac{2\vp_0}{q_*}\right).
\end{align}
By $\Phi(-\infty)=\mathbf1$, $\Phi(\infty)=\mathbf0$ and \eqref{1.b}, there exists a constant $C>1$ such that
\begin{equation}\label{VCC}
\begin{cases}
V(y,z)\geq (1-\vp q_*)\cdot\mathbf1 &\text{for all } z - m_*|y| \le -C,\\[4pt]
V(y,z)\leq \vp q_*\cdot\mathbf1 &\text{for all } z - m_*|y| \ge C.
\end{cases}
\end{equation}
Furthermore, there exists $\kappa>0$ such that
\begin{equation}\label{kapp}
\max_{i=1,\cdots,m}\left\{\sup_{-C \le z - m_*|y| \le C}(V_i)_z(y,z)\right\}=-\kappa<0.
\end{equation}
Take any
\begin{align}\label{sg}
0<\sigma<\min\left\{1,s\tilde\lambda_1,\frac{\vp}{2},\frac{\varpi}{2},\frac{\kappa}{2M\|\zeta\|_{L^\infty(\bar\Omega)}}\right\}.
\end{align}
Choose $\rho_1>0$ large enough such that
\begin{align}\label{rho1}
\frac{\rho_1\kappa}{2}
\geq{}&\left(\Lambda q^*+q^*+sM+\bar D\left(\frac{s^2}{c^2}+\|\psi''\|_{L^\infty(\mathbb R)}\right)M\right)\|\zeta\|_{L^\infty(\bar\Omega)}\\
&+2\bar D\frac{sM}{c}\|\nabla\zeta\|_{L^\infty(\bar\Omega)}
+\bar Dq^*\|\Delta\zeta\|_{L^\infty(\bar\Omega)}.\notag
\end{align}

By \eqref{1.5} and \eqref{1.7}, there exists $t_\sigma>0$ such that
\begin{align}\label{upp}
-\sigma q_*\cdot\mathbf{1}\leq u(t,x)-V(x_1,x_2-s t)\leq\sigma q_*\cdot\mathbf{1}
\quad\text{for all }|t|\geq t_\sigma\text{ and }x\in\bar\Omega.
\end{align}
Let $\rho_2>0$ be large enough such that
\begin{align}\label{rd}
\rho_2>C_K+c t_\sigma+1
\quad\text{and}\quad
d_1 e^{\tilde\lambda_1(C_K-\rho_2)}\leq\sigma q_* e^{-s\tilde\lambda_1 t_\sigma}.
\end{align}
For $t\geq -t_\sigma$ and $x\in\bar\Omega$, define functions $u^\pm=(u^\pm_1,\cdots,u^\pm_m)$ by
\[
\begin{cases}
u^+(t,x)=\min\left\{V(x_1,z^+)+\sigma Q(\xi^+(t,x))\zeta(x)e^{-\sigma (t+t_\sigma)},\mathbf{1}\right\},\\[4pt]
u^-(t,x)=\max\left\{V(x_1,z^-)-\sigma Q(\xi^-(t,x))\zeta(x)e^{-\sigma (t+t_\sigma)},\mathbf{0}\right\},
\end{cases}
\]
where
\[
z^\pm=x_2-st\mp\rho_1(1-e^{-\sigma (t+t_\sigma)})\mp\rho_2
\quad\text{and}\quad
\xi^\pm(t,x)=z^\pm-\psi(x_1),
\]
with $\psi$ defined as in \eqref{psii}.
We shall prove that the function $u^+$ is a supersolution of \eqref{1.1} in $[-t_\sigma,\infty)\times\bar\Omega$.
The fact that $u^-$ is a subsolution can be obtained by similar arguments.

Fix any index $i=1,\cdots,m$.
We first verify the initial and boundary conditions.
By \eqref{pq}, \eqref{zeta}, \eqref{upp} and $V_z(y,z)\ll\mathbf{0}$ for $(y,z)\in\R^2$, we have
\[
u^+(-t_\sigma,x)\geq\min\left\{V(x_1,x_2+s t_\sigma)+\sigma q_*\cdot\mathbf{1},\mathbf{1}\right\}
\geq u(-t_\sigma,x)
\quad\text{for }x\in\bar\Omega.
\]
For $t\geq -t_\sigma$ and $x\in\p\Omega$,
it follows from \eqref{psii}, \eqref{ck} and \eqref{rd} that 
\[
\xi^+(t,x)\leq z^+-m_*|x_1|\leq C_K-s t-\rho_2\leq C_K+s t_\sigma-\rho_2<0,
\]
hence $Q(\xi^+(t,x))=Q_1$ by \eqref{pq}.
By $\nu\cdot\nabla\zeta=1$ on $\partial\Omega$, \eqref{vps} and \eqref{V22}, one infers that 
\begin{align*}
\nu\cdot\nabla u_i^+(t,x)
\geq{}&-\|\nabla V_i(x_1,z^+)\|_{L^\infty(\R^2)}+\sigma q_i^1 e^{-\sigma(t+t_\sigma)}\\
&\geq -d_1 e^{\tilde\lambda_1(z^+-m_*|x_1|)}+\sigma q_* e^{-\sigma(t+t_\sigma)}\\
&\geq -d_1 e^{\tilde\lambda_1(C_K-st-\rho_2)}+\sigma q_* e^{-s\tilde\lambda_1(t+t_\sigma)}\\
&\geq0
\end{align*}
for all $t\geq -t_\sigma$ and $x\in\p\Omega$
such that $u_i^+(t,x)<1$.

It suffices to show that
$
\mathcal L_i(u^+)(t,x):
=(u_i^+)_t(t,x)-D_i\Delta u_i^+(t,x)-F_i(u^+(t,x))\geq0
$
for all $t\geq -t_\sigma$ and $x\in\bar\Omega$ such that $u_i^+(t,x)<1$.
By a direct calculation, one has
\begin{align*}
\mathcal L_i(u^+)(t,x)
={}&F_i(V(x_1,z^+))-F_i(u^+(t,x))
-\rho_1\sigma V_{x_2}(x_1,z^+)e^{-\sigma(t+t_\sigma)}\\
&+\sigma q_i'(\xi^+(t,x))(-s-\rho_1\sigma e^{-\sigma(t+t_\sigma)})\zeta(x) e^{-\sigma(t+t_\sigma)}\\
&-\sigma^2 q_i(\xi^+(t,x))\zeta(x)e^{-\sigma(t+t_\sigma)}
-D_i\sigma\Delta q_i(\xi^+(t,x))\zeta(x)e^{-\sigma(t+t_\sigma)}\\
&-2D_i\sigma\nabla q_i(\xi^+(t,x))\cdot\nabla\zeta(x)e^{-\sigma(t+t_\sigma)}
-D_i\sigma q_i(\xi^+(t,x))\Delta\zeta(x)e^{-\sigma(t+t_\sigma)}.
\end{align*}

If $\xi^+(t,x)\leq -C-m_0$, then $Q(\xi^+(t,x))=Q_1$ by \eqref{dpq}.
From \eqref{psii}, one has $z^+-m_*|x_1|\leq -C$.
By \eqref{vps} and \eqref{VCC}, one has
\[
(1-4\vp_0)\cdot\mathbf{1}\leq(1-\vp q_*)\cdot\mathbf{1}
\leq V(x_1,z^+)\leq u^+(t,x)\leq\mathbf{1}.
\]
By \eqref{muij}, one has
\[
F_i(V(x_1,z^+))-F_i(u^+(t,x))
\geq -\sigma\sum_{j=1}^m\mu_{ij}^1 q_j^1\zeta(x)e^{-\sigma (t+t_\sigma)}
\geq \sigma\varpi q_i^1\zeta(x)e^{-\sigma (t+t_\sigma)}.
\]
Since $V_z(y,z)\ll\mathbf{0}$ for $(y,z)\in\R^2$, one gets from \eqref{D}, \eqref{zeta} and \eqref{sg} that
\begin{align*}
\mathcal L_i(u^+)(t,x)
\geq{}&\sigma q_i^1\zeta(x)e^{-\sigma (t+t_\sigma)}
\left(\varpi-\sigma-\bar D\left\|\frac{\Delta\zeta}{\zeta}\right\|_{L^\infty(\bar\Omega)}\right)
\geq0.
\end{align*}

If $\xi^+(t,x)\geq C$, then $Q(\xi^+(t,x))=Q_0$ by \eqref{dpq}.
From \eqref{psii}, one has $z^+-m_*|x_1|\geq C$.
A similar argument as above implies that
$\mathcal L_i(u^+)(t,x)\geq0$.

If $-C-m_0\leq\xi^+(t,x)\leq C$, then $-C\leq z^+-m_*|x_1|\leq C$. 
By \eqref{pq} and \eqref{LA}, one has
\[
F_i(V(x_1,z^+))-F_i(u^+(t,x))
\geq -\sigma\Lambda q_i(\xi^+(t,x))\zeta(x)e^{-\sigma (t+t_\sigma)}
\geq -\sigma\Lambda q^*\|\zeta\|_{L^\infty(\bar\Omega)}e^{-\sigma (t+t_\sigma)}.
\]
By \eqref{naq} and \eqref{deq}, one has
$
|\nabla q_i(\xi^+)|\le \frac{sM}{c}
$ and
$
|\Delta q_i(\xi^+)| \le \left(\frac{s^2}{c^2} + \|\psi''\|_{L^\infty(\mathbb R)}\right) M$.
By \eqref{pq}, \eqref{D}, \eqref{kappa}, \eqref{sg} and \eqref{rho1}, there holds
\begin{align*}
\mathcal L_i(u^+)(t,x)
\geq {}&
-\sigma\Lambda q^*\|\zeta\|_{L^\infty(\bar\Omega)}e^{-\sigma (t+t_\sigma)}
+\rho_1\sigma \kappa e^{-\sigma(t+t_\sigma)}
-\sigma^2 q^*\|\zeta\|_{L^\infty(\bar\Omega)}e^{-\sigma(t+t_\sigma)}\\
&-\sigma sM\|\zeta\|_{L^\infty(\bar\Omega)}e^{-\sigma(t+t_\sigma)}
-\rho_1\sigma^2M \|\zeta\|_{L^\infty(\bar\Omega)}e^{-2\sigma(t+t_\sigma)}\\
&-\bar D\sigma\left(\frac{s^2}{c^2} + \|\psi''\|_{L^\infty(\mathbb R)}\right) M\|\zeta\|_{L^\infty(\bar\Omega)}e^{-\sigma(t+t_\sigma)}\\
&-2\bar D\sigma\frac{sM}{c}\|\nabla\zeta\|_{L^\infty(\bar\Omega)}e^{-\sigma(t+t_\sigma)}
-\bar D\sigma q^*\|\Delta\zeta\|_{L^\infty(\bar\Omega)}e^{-\sigma(t+t_\sigma)}\\[3pt]
\geq {}&
\sigma e^{-\sigma (t+t_\sigma)}
\Biggl(
\frac{\rho_1\kappa}{2}
+2\bar D\frac{sM}{c}\|\nabla\zeta\|_{L^\infty(\bar\Omega)}
+\bar Dq^*\|\Delta\zeta\|_{L^\infty(\bar\Omega)}\\
& 
-\left(\Lambda q^*+q^*+sM+\bar D\left(\frac{s^2}{c^2} + \|\psi''\|_{L^\infty(\R)}\right) M\right)\|\zeta\|_{L^\infty(\bar\Omega)}
\Biggr)\\[3pt]
\geq {}&0.
\end{align*}

Since $V(y,z)\to\mathbf{0}$ as $z - m_*|y|\to\infty$ and  
$V(y,z)\to\mathbf{1}$ as $z - m_*|y|\to-\infty$, it follows from the comparison principle, \eqref{pq} and \eqref{sg} that
there exist $K_1>0$ and $K_2>0$ large enough such that
\[
u(t,x)\geq u^-(t,x)
\geq \min\bigl(V(x_1,z^-)-\sigma q^*\cdot\mathbf{1},\mathbf{1}\bigr)
\geq (1-2\sigma q^*)\cdot\mathbf{1}
\geq (1-\vp q^*)\cdot\mathbf{1}
\]
for all $t\geq -t_\sigma$ and $x\in\bar\Omega$ such that $x_2-st-m_*|x_1|\geq K_1$, and
\[
u(t,x)\leq u^+(t,x)
\leq \max\bigl(V(x_1,z^+)+\sigma q^*\cdot\mathbf{1},\mathbf{0}\bigr)
\leq 2\sigma q^*\cdot\mathbf{1}
\leq \vp q^*\cdot\mathbf{1}
\]
for all $t\geq -t_\sigma$ and $x\in\bar\Omega$ such that $x_2-st-m_*|x_1|\leq -K_2$.
Therefore, by similar arguments as in Lemma~\ref{u-V},
one can get that there exists $\tilde R>0$ large enough such that
\begin{align}\label{vppp}
-\vp q^*\cdot\mathbf{1}
\leq u(t,x)-V(x_1,x_2-st)
\leq \vp q^*\cdot\mathbf{1}
\end{align}
for $t\geq -t_\sigma$ and $x\in\bar\Omega$ such that $|x|\geq \tilde R$.
Since $\vp$ was arbitrary,
one deduces from \eqref{upp} and \eqref{vppp} that \eqref{xinfty} holds.
This completes the proof of Theorem~\ref{t2}.
\end{proof}

\section*{Acknowledgment}
This paper was partially supported by NSF of China (12171120)  and by the Fundamental Research 
Funds for the Central Universities (No. 2022FRFK060028,  2023FRFK030022).

\section*{Data availability statements}
We do not analyse or generate any datasets, because our work proceeds within a theoretical and 
mathematical approach.

\section*{Conflict of interest}
There is no conflict of interest to declare.

\end{document}